\documentclass{amsart}

\usepackage{amsmath, amsthm}
\usepackage{amssymb}
\usepackage{subfiles}
\usepackage{comment}
\usepackage{cite}

\usepackage{tikz}
\usepackage{wrapfig}

\usepackage{todonotes}

\usetikzlibrary{positioning}

\theoremstyle{definition}
\newtheorem{thm}{Theorem}[section]
\newtheorem*{thm*}{Theorem}
\newtheorem{lemma}[thm]{Lemma}

\newtheorem{defn}[thm]{Definition}
\newtheorem{claim}[thm]{Claim}

\newtheorem{prop}[thm]{Proposition}
\newtheorem{cor}[thm]{Corollary}

\newtheorem{remark}[thm]{Remark}
\newtheorem{fact}[thm]{Fact}

\newtheorem{ex}[thm]{Example}

\newtheorem{obs}[thm]{Observation}

\renewcommand{\subset}{\subseteq}

\newcommand\force{\Vdash}
\newcommand\R{\mathbb{R}}
\newcommand\T{\mathbb{T}}

\newcommand\Q{\mathbb{Q}}
\newcommand\B{\mathbb{B}}
\renewcommand\P{\mathbb{P}}

\DeclareMathOperator{\dom}{dom}
\DeclareMathOperator{\supp}{supp}
\DeclareMathOperator{\hod}{HOD}
\DeclareMathOperator{\range}{range}
\DeclareMathOperator{\Col}{Col}
\DeclareMathOperator{\Add}{Add}

\DeclareMathOperator{\proj}{proj}
\DeclareMathOperator{\len}{len}
\DeclareMathOperator{\trcl}{trcl}
\DeclareMathOperator{\type}{type}
\DeclareMathOperator{\sym}{sym}

\newcommand{\Ord}{\mathrm{Ord}}
\newcommand{\ZFC}{\mathrm{ZFC}}
\newcommand{\ZF}{\mathrm{ZF}}
\newcommand{\DC}{\mathrm{DC}}
\newcommand{\AC}{\mathrm{AC}}
\newcommand{\KW}{\mathrm{KW}}

\newcommand{\set}[2]{ \left\{ #1 :\, #2 \right\} }
\newcommand{\seqq}[2]{ \left\langle #1 :\, #2\right\rangle }
\newcommand{\seq}[1]{ \left\langle #1 \right\rangle }

\title{Intermediate models with deep failure of choice}
\date{\today}

\author[Yair Hayut]{Yair Hayut}
\address{Institute of Mathematics, The Hebrew University of Jerusalem, Jerusalem 91904, Israel}
\email{yair.hayut@mail.huji.ac.il}
\author[Assaf Shani]{Assaf Shani}
\address{Department of Mathematics and Statistics, Concordia University University, Montreal, QC  H3G 1M8, Canada}
\email{assaf.shani@concordia.ca}
\thanks{The research of the first author was partially supported by the The Lise-Meitner grant, FWF, 2650-N35 and the Israel Science Foundation grant number 1967/21. The research of the second author was partially supported by NSF grant DMS-2246746 and NSERC grant RGPIN-2024-05827.}


\begin{document}
\begin{abstract}
The following question was asked by Grigorieff~\cite{Grigorieff-1975}. Suppose $V$ is a ZFC model and $V[G]$ is a set-generic extension of $V$. Can there be a ZF model $N$ so that $V\subset N \subset V[G]$ yet $N$ is not equal to $V(A)$ for any set $A\in V[G]$?
The first such model appeared in~\cite{Karagila-Bristol-model-2018}. This is the so-called \emph{Bristol model}, an intermediate model between $L$ and $L[c]$ where $c$ is a Cohen-generic real over $L$. Karagila~\cite{Karagila-Bristol-model-2018} further proves that the Kinna-Wager degree is unbounded in this model.

We prove that such an intermediate extension can be found in a Cohen-generic extension of \emph{any} ground model, fully resolving Grigorieff's question. That is, let $V$ be \emph{any} $\ZF$ model and $c$ a Cohen-generic real over $V$. We prove that there is an intermediate ZF-model $V\subset N \subset V[c]$ so that
\begin{itemize}
    \item $N$ is not equal to $V(A)$ for any set $A\in V[c]$;
    \item The Kinna-Wagner degree of $N$ is unbounded;
    \item No set forcing in $N$ recovers the axiom of choice.
\end{itemize}
\end{abstract}
\maketitle
\section{Introduction}
Given a model $V$ of ZFC, and a set generic extension $V[G]$, what are the intermediate models $V\subset M \subset V[G]$ of ZF? 
The more common ones are of the form $V(A)$, the minimal transitive extension of $V$ which contains the set $A$ and satisfies ZF, where $A$ is a set in $V[G]$. For example, all intermediate models of ZFC are of this form, as well as the familiar symmetric models used for various independence results over ZF.
Grigorieff~\cite{Grigorieff-1975} and Usuba~\cite{Usuba-2021} studied such intermediate extensions extensively and characterized them in several ways. For example, Grigorieff~\cite[Theorem~B]{Grigorieff-1975} proved that $M$ is of the form $V(A)$ if and only if $M[G]$ can be recovered as a set generic extension of $M$. Usuba~\cite{Usuba-2021} proved that intermediate extensions of the form $V(A)$ are exactly the symmetric extensions of $V$ using some forcing notion (not necessarily a projection of the one used to obtain $G$). 

Grigorieff asked \cite[p. 471]{Grigorieff-1975} if there can be an intermediate extension not of that form. That is, $V\subset M\subset V[G]$ yet $M$ is not $V(A)$ for any set $A$ in $V[G]$.

The first such intermediate extension was given in \cite{Karagila-Bristol-model-2018}, the so-called Bristol model. Specifically, this is an intermediate model $L\subset M\subset L[c]$ where $c$ is a Cohen-generic real over G\"odel's constructible universe $L$. 
The main drawback of the Bristol model is its limitation on the ground model. The construction relies on the existence of certain scales, which exist in $L$ but are incompatible with large cardinals such as supercompact cardinals.

The following question remained open: can sufficiently large cardinals, such as supercompact or extendible cardinals, prohibit the existence of such exotic intermediate extensions? 
This is particularly motivated by results of Woodin~\cite{Woodin-SEM1-2010}, showing that large cardinals outright imply certain fragments of choice, and results of Usuba~\cite{Usuba-2021}, showing that large cardinals significantly restrict the structure of ZFC submodels. Some recent works of Goldberg, \cite{Goldberg2024}, show that large cardinals that are inconsistent with the axiom of choice also impose choice-like consequences on the universe.

The main result of this paper is to construct an exotic intermediate extension starting with any ground model, giving a complete resolution to Grigorieff's question.
\begin{thm}\label{thm: intermediate extension}
Let $V$ be \emph{any} model of $\ZF$, and $c$ a Cohen-generic real over $V$. Then there is an intermediate model $V\subset M\subset V[c]$ such that $M \neq V(A)$ for any set $A\in V[G]$.
\end{thm}

\subsection{Deep failure of choice}
If $V$ is a model of ZFC, and $M$ an intermediate ZF-extension $V \subset M \subset V[G]$, being of the form $V(A)$ for some $A\in V[G]$ can be seen as a weak fragment of choice. We consider a model as in Theorem~\ref{thm: intermediate extension} as satisfying a ``deep'' failure of choice. 

Another measure of how deeply choice fails in a ZF model $M$ is the small violations of choice principle (SVC) introduced by Blass~\cite{BlassSVC}. In this paper, the most relevant aspect of SVC is the equivalent assertion that the axiom of choice can be forced over $M$ using a set forcing (see \cite[Theorem~4.6]{BlassSVC}).

Yet another way of measuring how deeply choice fails is the Kinna-Wagner degree.
\begin{defn}
Let $M$ be a model of $\ZF$. The Kinna-Wagner degree of $M$ is the least ordinal $\alpha$ such that for every $X \in M$ there is an injection $f \in M$ from $X$ to $\mathcal{P}^\alpha(\Ord)$ (if there is one). We say that the Kinna-Wagner degree of a model is unbounded, or $\infty$, if it is not of degree $\alpha$ for any ordinal $\alpha \in M$.
\end{defn}
The Kinna-Wagner principle, which states that the Kinna-Wagner degree is $\leq 1$, was introduced by Kinna and Wagner \cite{Kinna-Wagner-1955}, where they proved that it is equivalent to the following selection principle: suppose $X$ is a set whose members are sets with at least two elements, then there is a function $f\colon X\to\mathcal{P}(X)$ such that $f(x)$ is a non-empty proper subset of $x$, for each $x\in X$.

Jech~\cite{Jech-models-without-AC-1971} constructed a model of ZF in which the Kinna-Wagner principle fails. Monro~\cite{Monro-1973} introduced the higher Kinna-Wagner principles, and constructed, for each finite $n\in\omega$, a model of ZF in which the Kinna-Wagner degree is $\geq n$.
Monro's results were extended in \cite{Karagila-iterating-2019} to find a model of $\ZF$ with $\KW$ degree equal to $\omega$.

Karagila~\cite{Karagila-Bristol-model-2018} proved that the Bristol model has an unbounded Kinna-Wagner degree, and this was the first known such model. Moreover, SVC fails in the Bristol model.
\begin{thm}
   In the exotic intermediate extension $M$ from Theorem~\ref{thm: intermediate extension}:
   \begin{itemize}
       \item the Kinna-Wagner degree is unbounded;
       \item no set forcing in $M$ recovers the axiom of choice.
   \end{itemize}
\end{thm}

\subsection{A sketch of the construction}
The first construction of models of ZF with increasingly ``deeper''  failures of choice (in the sense of the Kinna-Wagner degree) was done by Monro~\cite{Monro-1973}.
Starting with a model $V$ of ZFC, let $V(A_1)$ be the ``basic Cohen model'' generated by an unordered set $A_1$ of Cohen reals. Here we consider a Cohen real as a generic subset of $\omega$, added by finite approximations.
Monro then ``repeated'' this construction by adding an unordered set $A_2$ of infinitely many generic subsets of $A_1$, again added by finite approximations.
Continuing this way, we generically add $A_3, A_4$, and so on.

Monro then shows that in the model $V(A_n)$ generated by $A_n$ over $V$, the Kinna-Wagner degree is $\geq n$. Furthermore, the model $V(A_n)$ is not generated by any set $B$ if lower rank than $A_n$ (see \cite{Shani-2021}).

A key property of Monro's construction is that the \emph{lower rank sets are eventually stabilized}. Specifically, for $n\leq m$, $V(A_n)$ and $V(A_m)$ have the same sets of rank $\leq n$.
This property is similarly crucial in the constructions in \cite{Karagila-Bristol-model-2018, Shani-2021}, and will be in our construction below.

The main difficulty is marching past stage $\omega$. A direct attempt to add a new set of rank $\omega$, using Monro's trick, would be to add a generic subset to $A_\omega=\bigcup_{n<\omega}A_n$. However, this would add generic subsets to each $A_n$, therefore adding lower rank sets.
A more reasonable attempt would be to add new sets of rank $\omega$ as choice functions in $\prod_n A_n$. However, given two such mutually generic choice functions, the set of indices $n\in\omega$ on which they agree would be a new subset of $\omega$, again adding a new lower rank set.

This problem of passing the limit stages is indeed the main one. The first construction of a model where the $\omega$'th Kinna-Wagner principle fails is the Bristol model, in which the Kinna-Wagner rank is $\infty$. 
To deal with the limit stages, the Bristol model construction uses heavy assumptions on the ground model $V$, the existence of certain scales that are inconsistent with large cardinals. 

The key tool, allowing us to continue a Monro-like construction throughout all the ordinals, is introducing a generic tree structure on the ``Monro sets''.
For example, we add a generic tree structure on the sets $A_n$, $n<\omega$, where $A_n$ is level $n$ of the tree. 
Then, we add sets of rank $\omega$ simply by adding infinitely many branches through this tree. 
This settles the issue mentioned above: any two distinct branches will be eventually different, and therefore will not introduce a new subset of $\omega$ (at least not in the trivial manner described above).

Another key aspect of our tree structure is its uniformity. At each stage we only add ``on top'', so that at limit stage $\theta$ we naturally get a tree of height $\theta$ (without branches), satisfying the desired properties, allowing us to add generic branches through it without disturbing lower rank sets.

At the end of this construction, we arrive at a model of $\ZF$ together with sets $A_\alpha$, of increasing rank, for all ordinals $\alpha$, and a tree structure on those. 
Furthermore, the members of each $A_\alpha$ satisfy sufficient indiscernibility over the lower rank sets (see Definition~\ref{defn: tree indiscernibility}). This implies in particular that choice cannot be regained without collapsing all the $A_\alpha$'s, which also implies the unbounded Kinna-Wagner degree of the resulting model.


In Section~\ref{Section: preliminaries} we make a few basic remarks on the Kinna-Wagner degree and other notions of ``deep failure of choice''.

In Section~\ref{subsection: stepup} we develop the basic tools that will be used for the successor step of our construction. This involves two sub-steps: one is adding a set of higher rank, which we do as a variation of ``Monro's step''; the second is adding the tree structure, making this newly added set the next level of the tree.  

In Section~\ref{section:construction} we present the \emph{basic construction} of our model $M = V(A,T)$ (where $A$ and $T$ are classes), as an inner model of a class generic extension of $V$, for any ground model $V$ of $\ZF$. In Section~\ref{subsection: full indiscernibility} we further analyze this model.

In Section~\ref{section:shuffling} we prove some technical results regarding generic permutations of partially generic filters.

Finally, in Section~\ref{section: construction in Cohen extension} we show that the model $M = V(A,T)$ from Section~\ref{section:construction} can be constructed inside a single Cohen real extension of $V$, $V[c]$.


\section{Preliminaries}\label{Section: preliminaries}
Given a model of ZF, we may measure how ``far'' it is from satisfying AC according to its Kinna-Wagner degree. We can also ask whether, and what kind of, forcing can recover AC (see~\cite{BlassSVC}). We make here a few remarks on these notions and the relationship between them.



\begin{lemma}[Folklore]
Let $V$ be a model of $\ZF$ and let $\mathbb{P}\in V$ be a well ordered forcing notion such that $\Vdash_{\mathbb{P}} \AC$ then $V \models \AC$.
\end{lemma}
\begin{proof}
Pick $x \in V$ and let $\tau$ be a $\mathbb{P}$-name for a bijection between some ordinal $\alpha$ and $x$. Let $p\in \mathbb{P}$ be a condition that forces this. Then, the partial map from $\mathbb{P} \times \alpha$ to $x$ which is defined by $(p, \beta) \mapsto y$ iff $p \Vdash \tau(\check{y}) = \check{\beta}$, is a surjection from a well ordered set onto $x$ in $V$.
\end{proof}
\begin{lemma}[$\ZF$]
Let $\kappa$ be a limit ordinal,
\[\Col(\omega, <V_\kappa) \times \Col(\omega, \kappa) \cong \Col(\omega, V_\kappa).\] 
\end{lemma}
\begin{proof}
The forcing $\Col(\omega, <V_\kappa)$ adds a sequence of generic functions $\langle g_\alpha \mid \alpha < \kappa\rangle$, $g_\alpha \colon \omega \to V_\alpha$, which are onto and infinite-to-one. The forcing $\Col(\omega, \kappa)$ adds a function $f \colon \omega \to \kappa$ which is onto and infinite-to-one. Given these, define $h\colon \omega\to V_\kappa$ by sending $n$ to $g_\alpha(m)$ if $f(n) = \alpha$ and $m = |\{k < n \mid f(k) = \alpha\}|$. Then $h$ is onto and infinite-to-one.

The above description gives a forcing isomorphism from a dense subset of $\Col(\omega, <V_\kappa) \times \Col(\omega, \kappa)$ to $\Col(\omega, V_\kappa)$. That is, for a dense set of conditions $(p, q) \in \Col(\omega, <V_\kappa) \times \Col(\omega, \kappa)$ there is a well defined $r \in \Col(\omega, V_\kappa)$, defined from $p$ and $q$ as $h$ is defined from $\langle g_\alpha \mid \alpha < \kappa\rangle$ and $f$ above, and the map $(p,q)\mapsto r$ is a forcing isomorphism onto a dense subset of $\Col(\omega, V_\kappa)$. 
\end{proof}
Since $\Col(\omega,\kappa)$ is always well ordered, we conclude that question about choice principles in the extension by $\Col(\omega, V_\kappa)$ and $\Col(\omega, {<}V_\kappa)$ are typically equivalent. 

\begin{lemma}\label{lemma: bound on KW from forcing AC}
Let $V$ be a model of ZF and $\P\in V$ a poset which can be embedded in $\mathcal{P}^\kappa(\mathrm{Ord})$ for a limit ordinal $\kappa$. If $\force_{\P}\AC$ then the Kinna-Wagner degree of $V$ is $\leq \kappa$. 
\end{lemma}
\begin{proof}
We may assume that $\mathbb{P}$ is a subset of $\mathcal{P}^\kappa(\eta)$, for some ordinal $\eta$.
Fix a set $X$. We need to find an injective map, in $V$, from $X$ to $\mathcal{P}^{\kappa}(\mathrm{Ord})$. Let $\tau\in V$ be a $\P$-name for an injective map from $\check{X}$ to $\alpha$, for some ordinal $\alpha$. Given $x\in X$, let $\gamma(x)$ be the minimal ordinal $\gamma<\kappa$ so that there is some $p\in \mathcal{P}^\gamma(\eta)$ which decides the value of $\tau(\check{x})$. Define 
\begin{equation*}
    f(x) = \set{(p,\zeta)}{p\in \mathcal{{P}}^{\gamma(x)}(\eta),\, p\force \tau(\check{x}) = \zeta}.
\end{equation*}
Then, in $V$, $f$ is an injective map from $X$ to $\mathcal{P}^\kappa(\eta) \times \alpha$. Note that, provably in ZF, $\mathcal{P}^\kappa(\eta)\times\alpha$ is embeddable into $\mathcal{P}^{\kappa}(\mathrm{Ord})$.
\end{proof}
We conclude that if $\Col(\omega, V_\kappa)$ forces choice, then the Kinna-Wagner degree is bounded by $\kappa$.

The following remark says that the statement ``there is an injective map from $X$ to $\mathcal{P}^\alpha(\mathrm{Ord})$'' is essentially a statement about $\mathcal{P}^{\alpha+1}(\mathrm{Ord})$.
\begin{remark}\label{remark: coding injective map}
    Let $X$ be a transitive set and $f\colon X \to \mathcal{P}^\alpha(\mathrm{Ord})$ injective. Then there is a well-founded relation $R$ on a subset of $\mathcal{P}^\alpha(\mathrm{Ord})$ which is isomorphic to $(X,\in)$. In this case, the Mostwoski collapse of $R$ is $X$, and therefore its inverse is an injective map from $X$ to $\mathcal{P}^\alpha(\mathrm{Ord})$, which is definable from $R$.

    Using a definable bijection $\mathrm{Ord}\times \mathrm{Ord} \to \mathrm{Ord}$, we have a definable bijection $\mathcal{P}^\alpha(\mathrm{Ord}) \times \mathcal{P}^\alpha(\mathrm{Ord}) \to \mathcal{P}^\alpha(\mathrm{Ord})$, and so we may think of $R$ as a subset of $\mathcal{P}^\alpha(\mathrm{Ord})$.
\end{remark}

\subsection{Models generated by a set}

\begin{defn}
Let $V$ be a $\ZF$-model, $V'$ an outer $\ZF$-model with the same ordinals, and $A\in V'$.
The model
\begin{equation*}
    V(A)=\bigcup_{\alpha \in \Ord} L_\alpha(V_\alpha \cup \trcl(\{A\}))
\end{equation*}
is the minimal transitive model of $\ZF$ containing $V \cup \{A\}$.
\end{defn}

Note that $V(A)$ does not depend on the outer model $V'$.
By minimality, it follows that
\[V(A)=\mathrm{HOD}^{V(A)}_{V,\trcl(\{A\})},\]
the universe of sets that are hereditarily definable using parameters in $V$ and the transitive closure of $\{A\}$, as computed inside $V(A)$.
That is, in $V(A)$, any set is definable using parameters in $V$ and the transitive closure of $A$:
\begin{fact}\label{fact: V(A) model}
Let $X\in V(A)$. There is a formula $\psi$, parameter $v\in V$ and parameters $\bar{a}$ from the transitive closure of $A$, so that $X$ is the unique solution to $\psi(X,A,\bar{a},v)$ in $V(A)$.
Equivalently, there is a formula $\phi$ so that $X$ is defined in $V(A)$ as the set of all $x$ so that $\phi(x, A,\bar{a},v)$ holds.
\end{fact}

Below we will often work with models of the form $V(A)$ where $A \in V[G]$, a generic extension of $V$. Often times the model $V(A)$ is in fact equal to \[\hod^{V[G]}_{V, \trcl(\{A\})}.\] This is generally the case for the models constructed in Section~\ref{section:construction} below (see Theorem~\ref{thm: M_alpha is symmetric extension}).

This is not always the case, and the distinction is important. A theorem of Gregorioff~\cite[Theorem C (ii)]{Grigorieff-1975} states that models of the form $\hod^{V[G]}_{V, \trcl(\{A\})}$ are precisely the symmetric submodels of $V[G]$. On the other hand, a theorem of Usuba~\cite[Lemma 4.8]{Usuba2021geology} states that $V(A)$ is always a symmetric submodel of \emph{some} generic extension of $V$, possibly not with the forcing we used to add $G$. 

In particular, in Section~\ref{section: construction in Cohen extension} we work in $V[c]$, where $c$ is a Cohen generic real over $V$, and construct class-many distinct intermediate extensions of the form $V(A)$. Since there are only set-many symmetric submodels of any given set forcing, these cannot all be of the form $\hod^{V[G]}_{V, \trcl(\{A\})}$.


\subsection{A basic permutation argument}
Let $V$ be a model of $\ZF$, $\mathbb{Q}\in V$ a poset, $I\in V$ an infinite index set, and let $\mathbb{P}$ be the finite support product of $I$-many copies of $\Q$.
We identify a $\P$-generic filter $G$ over $V$ with the corresponding indexed family $\seqq{G(i)}{i\in I}$ of $\Q$-generic filters over $V$.
Let $\dot{A}$ be the name for the unordered set of generics $\set{\dot{G}(i)}{i\in I}$.
\begin{remark}
In the applications below, our $\Q$-generics will be identified in a definable way with some other objects. For example, if $\Q$ is Cohen forcing for adding a single real, then we identify each $G(i)$ with the corresponding Cohen real and similarly identify $A$ with the set of the Cohen reals.
\end{remark}
The following lemma is
at the heart of the analysis of the basic Cohen model. See \cite[p. 133]{Felgner-ZF-set-thery-1971} or \cite[Proposition 1.2]{Blass-1981} in the context of the basic Cohen model, or \cite[Lemma 2.4]{Shani-2021}.
\begin{lemma}\label{lemma: basic permutation}
Fix a $\P$-generic $G$ over $V$ and let $A=\set{G(i)}{i\in I}$.
Let $\bar{a}=\seq{a_1,\dots,a_n}$ be a finite sequence of distinct members of $A$, $\phi$ a formula and $v\in V$ a parameter. Note that $\bar{a}$ is $\Q^n$-generic over $V$. Then there is a formula $\psi$ such that
\begin{equation*}
    \phi^{V(A)}(A,\bar{a},v)\iff \psi^{V[\bar{a}]}(\P,\bar{a},v)
\end{equation*}
\end{lemma}
\begin{proof}
Let $k_1,\dots,k_n\in I$ such that $a_i=G(k_i)$. Let $\psi(\P,\bar{a},v)$ be the formula: ``$\exists p\in\P$ such that $p(k_i)\in a_i$ for $i=1,\dots,n$ and $p\force_\P\phi^{V(\dot{A})}(\dot{A},\dot{G}(k_1),\dots,\dot{G}(k_n),\check{v})$''. The following claim implies that \[\phi^{V(A)}(A,\bar{a},v)\iff \psi^{V[\bar{a}]}(\P,\bar{a},v),\] concluding the proof of the lemma.
\begin{claim}
Suppose $p,q\in\P$ are conditions which agree on $\bar{a}$, that is, $p(k_i),q(k_i)$ are in $a_i$ for $i=1,\dots,n$. Then $p,q$ cannot force different truth values for the statement $\phi^{V(\dot{A})}(\dot{A},\dot{G}(k_1),\dots,\dot{G}(k_n),\check{v})$.
\end{claim}
\begin{proof}
Otherwise, we may find a generic $G$ such that \[V[G] \models \phi^{V(A)}(A,G(k_1),\dots,G(k_n),v),\] and some condition $q$ which agrees with $G(k_1),\dots,G(k_n)$, such that \[q\force\neg\phi^{V(\dot{A})}(\dot{A},\dot{G}(k_1),\dots ,\dot{G}(k_n),\check{v}).\]
By applying a finite permutation of $I$ fixing $k_1,\dots,k_n$, find a generic $\tilde{G}$ such that
\begin{itemize}
    \item $\tilde{G}(k_i)=G(k_i)=a_i$;
    \item $\dot{A}[G]=\set{G(i)}{i\in I}=\set{\tilde{G}(i)}{i\in I}=\dot{A}[\tilde{G}]$;
    \item the condition $q$ is in the generic $\tilde{G}$.
\end{itemize}
Working now in the extension $V[\tilde{G}]$ (which is the same model as $V[G]$), we conclude that $\phi^{V(A)}(A,G(k_1),\dots,G(k_n),v)$ fails, a contradiction.
\end{proof}
\end{proof}

\section{Step up}\label{subsection: stepup}
In this section, we describe a construction to increase the Kinna-Wagner degree, similar to, but different than, Monro's \cite{Monro-1973}.
The definitions and results in this section will be crucial to the ``successor step'' of our construction below. In particular, the posets $\Q(W)$ and $\T(A, W)$ will be used later.

\begin{defn}
Let $V$ be a model of ZF and $W$ a set in $V$. Let $\Q=\Q(W)$ be the poset of finite partial functions from $\omega\times W$ to $\{0,1\}$, ordered by reverse extension.
\end{defn}
The $\Q$-generic object is identified with a sequence $\seqq{x_n}{n<\omega}$ where each $x_n$ is a subset of $W$, by identifying the set $x_n$ with its characteristic function.
\begin{center}
Let $[x_n]=\set{y\subset W}{y\Delta x_n\textit{ is finite}}$, and define $A=\set{[x_n]}{n\in\omega}$.    
\end{center}

\begin{lemma}[Basic step]\label{lemma: basic step}
The members of $A=\set{[x_n]}{n\in\omega}$ are indiscernible over parameters in $V$.
That is, for any formula $\phi$ and parameter $v\in V$, if $\Bar{a},\Bar{b}\subset A$ are finite tuples from $A$ of the same length such that $a_i=a_j\iff b_i=b_j$ for any $i,j<n$, then
\begin{equation*}
    V(A)\models    \phi(A,\Bar{a},v)\iff\phi(A,\Bar{b},v).
\end{equation*}
Moreover, the type of $\Bar{a}$ over $V$ is (uniformly) definable in $V$. That is, for any formula $\phi(A,\Bar{a},v)$ there is a formula $\psi(W,v)$ such that
\begin{equation*}
    V\models\psi(W,v)\iff V(A)\models\phi(A,\Bar{a},v).
\end{equation*}
\end{lemma}
\begin{proof}
It suffices to prove the statement for tuples of distinct elements from $A$.
Let $\bar{a}=\seq{a_0,\dots,a_{n-1}}$ be a sequence of distinct $n$ members of $A$, $a_i=[x_{k_i}]$. Fix a formula $\phi$ and parameter $v\in V$.
For any condition $p\in\Q$, there is a $\Q$-generic over $V$, $\seqq{y_n}{n<\omega}$, such that
\begin{itemize}
    \item $y_n\Delta x_n$ is finite for each $n<\omega$;
    \item $\seqq{y_n}{n<\omega}$ extends $p$.
\end{itemize}
By the first condition, the set $A$ computed by $\seqq{y_n}{n<\omega}$ is the same set computed by $\seqq{x_n}{n<\omega}$. 
It follows that
\begin{equation*}
    V(A)\models \phi(A,\bar{a})\iff V\models \Q\force\phi(A,[\dot{x}_{k_0}],\dots,[\dot{x}_{k_{n-1}}],v).
\end{equation*}
Furthermore, for any $\Q$-generic $\vec{x}=\seqq{x_n}{n<\omega}$ over $V$ there is a $\Q$-generic $\vec{y}=\seqq{y_n}{n<\omega}$ over $V$ such that 
\begin{itemize}
    \item $\set{[y_n]}{n\in\omega}=\set{[x_n]}{n\in\omega}$;
    \item $y_i=x_{k_i}$.
\end{itemize}
It follows that, in $V$, \[\Q\force\phi(A,[\dot{x}_{k_0}],\dots,[\dot{x}_{k_{n-1}}],v)\iff \Q\force\phi(A,[\dot{x}_0],\dots,[\dot{x}_{n-1}],v).\]
Finally, let $\psi(W,v)$ be the statement $\Q\force\phi(\dot{A},[\dot{x}_0],\dots,[\dot{x}_{n-1}],v)$ (note that the poset $\Q$ is definable from $W$). Then for any sequence $\bar{a}=a_0,\dots,a_{n-1}$ of distinct members of $A$,
\begin{equation*}
    V\models\psi(W,v)\iff V(A)\models\phi(A,\bar{a},v).
\end{equation*}
\end{proof}

If $V$ is a choiceful model and $W=\omega$, then the model $V(\set{x_n}{n\in\omega})$, generated by the unordered set of Cohen reals, is essentially the \textit{basic Cohen model} (see \cite{Kanamori-Cohen-08} or \cite{Jech2003}). A weaker form of indiscernibility holds in this model (see the Continuity Lemma \cite[p.133]{Felgner-ZF-set-thery-1971}). Similar forms of indiscernibility hold in Monro's models (see \cite[Lemma 7.2]{Shani-2021}). 
The move from $x_n$ to the equivalence classes $a_n$ provides full indiscernibility and simplifies our construction below.

If one repeats the previous basic step twice, to find $A_1$, a symmetric $\Q(W)$-generic over $V$, and then to find $A_2$, a symmetric $\Q(A_1)$-generic over $V(A_1)$, then in the model $V(A_1)(A_2)$, the members of $A_1$ and $A_2$ are completely indiscernible over one another.
We will want to limit the indiscernibility, in a controlled way, by introducing a generic tree structure between the levels.

\begin{defn}
Following the notation of the previous lemma, working in $V(A)$, let $\T=\T(A, W)$ be the poset of all finite functions from $A$ to $W$, ordered by reverse inclusion. \end{defn}
The generic object is identified with a surjection $\pi\colon A\to W$, which we identify with a tree structure $T$ between $W$ and $A$. Specifically, for $w\in W$ and $a\in A$, $w<_T a$ if $\pi(a)=w$. Note that for each $w\in W$ there are infinitely many $a\in A$ above it in the tree.

\begin{lemma}[Basic tree step]\label{lemma: basic tree step}
The type of members of $A$ over $V$ is determined by the tree structure. That is, for any formula $\phi$ and parameter $v\in V$, for any $\Bar{a},\Bar{b}\subset A$, if
\begin{itemize}
    \item $\pi(a_i)=\pi(b_i)$ for all $i<n$, and
    \item $a_i=a_j\iff b_i=b_j$ for any $i,j<n$, then
\end{itemize}
\begin{equation*}
    V(A,T)\models    \phi(A,T,\Bar{a},v)\iff\phi(A,T,\Bar{b},v).
\end{equation*}
Moreover, the type of $\Bar{a}$ over $V$ is definable in $V$ from $W$ and the sequence $\bar{w}=\seqq{w_i}{i<n}$ where $w_i=\pi(a_i)$. That is, for any formula $\phi(A,\Bar{a},v)$ there is a formula $\psi(W,\bar{w},v)$ such that for any $\bar{a}$, if $w_i=\pi(a_i)$ then
\begin{equation*}
    V\models\psi(W,\bar{w},v)\iff V(A,T)\models\phi(A,T,\Bar{a},v).
\end{equation*}
\end{lemma}
\begin{proof}
Given a tuple $\bar{a}=\seq{a_0,\dots,a_{n-1}}$ of distinct elements from $A$ and a tuple $\bar{w}=\seq{w_0,\dots,w_{n-1}}$ from $W$, let $t[\bar{a},\bar{w}]\in\T$ be the condition with domain $\set{a_i}{i<n}$ such that $t[\bar{a},\bar{w}](a_i)=w_i$.
\begin{claim}
For any formula $\phi$ and parameter $v\in V$, for any tuple $\bar{a}$ of distinct elements in $A$, let $w_i=\pi(a_i)$, then
\begin{equation*}
    V(A,T)\models \phi(A,T,\bar{a},v)\iff V(A)\models t[\bar{a},\bar{w}]\force \phi(A,\dot{T},\bar{a},v).
\end{equation*}
\end{claim}
\begin{proof}
To prove the claim, we assume there is some extension $p$ of $t[\bar{a},\bar{w}]$ which forces $\phi(A,\dot{T},\bar{a},v)$, and show that $t[\bar{a},\bar{w}]$ already forces it. 

Let $\bar{a},\bar{b}$ be an enumeration of the domain of $p$, where $\bar{b}=\seq{b_0,\dots,b_{m-1}}$ are distinct members of $A$, not appearing in $\bar{a}$.
Let $w_i=\pi(a_i)$ for $i<n$ and $u_i=\pi(b_i)$ for $i<m$.
Then, in $V(A)$, the statement $p\force\phi(A,\dot{T},\bar{a},v)$ can be expressed by a formula $\chi(A,\bar{a},\bar{b},v,W,\bar{w},\bar{u})$.

By Lemma~\ref{lemma: basic step}, for any sequence $\bar{b}'=\seq{b'_0,\dots,b'_{m-1}}$ of distinct members of $A$, not appearing in $\bar{a}$, $\chi(A,\bar{a},\bar{b}',v,W,\bar{w},\bar{u})$ holds as well.
This means that the condition $p[\bar{b}']$ forces $\phi(A,\dot{T},\bar{a},v)$, where $p[\bar{b}']$ is the condition defined on the domain $\bar{a},\bar{b}'$ so that $p[\bar{b}'](a_i)=w_i$ and $p[\bar{b}'](b'_i)=u_i$. 

Finally, note that for any condition $q\in\T$, if $q$ extends $t[\bar{a},\bar{w}]$ then $q$ is compatible with $p[\bar{b}']$, for $\bar{b}'$ disjoint from the domain of $q$. So no condition extending $t[\bar{a},\bar{w}]$ can force the negation of $\phi(A,\dot{T},\bar{a},v)$. 
By the forcing theorem, $t[\bar{a},\bar{w}]\force \phi(A,\dot{T},\bar{a},v)$.
\end{proof}

Fix a formula $\phi$ and a parameter $v\in V$.
Working in $V(A)$, for a finite tuple $\bar{w}$ from $W$, let $\zeta(A,\bar{a},\bar{w},v)$ be the statement
\begin{equation*}
    t[\bar{a},\bar{w}]\force_\T \phi(A,\dot{T},\bar{a},v).
\end{equation*}
By the claim above, for any $\bar{a}$, if $w_i=\pi(a_i)$, then $\phi(A,T,\bar{a},v)$ holds in $V(A,T)$ if and only if $\zeta(A,\bar{a},\bar{w},v)$ holds in $V(A)$.
Using Lemma~\ref{lemma: basic step}, let $\psi(W,\bar{w},v)$ be a formula such that
\begin{equation*}
    V\models \psi(W,\bar{w},v)\iff V(A)\models\zeta(A,\bar{a},\bar{w},v).
\end{equation*}
Finally, for any $\bar{a}$ and $\bar{w}$ so that $w_i=\pi(a_i)$,
\begin{equation*}
        V\models \psi(W,\bar{w},v)\iff V(A)\models\zeta(A,\bar{a},\bar{w},v)\iff V(A,T)\models\phi(A,T,\bar{a},v).
\end{equation*}
This implies the first conclusion of the lemma, that if $\bar{b}$ is another tuple with $\pi(b_i)=\pi(a_i)=w_i$, then 
\begin{equation*}
    V(A,T)\models\phi(A,T,\bar{a},v)\iff \phi(A,T,\bar{b},v).
\end{equation*}
\end{proof}

\begin{remark}
For any finite $\bar{a}\subset A$, we may replace the base model $V$ with $V(\bar{a})$. If $a_i=[x_{k_i}]$, $i=0,\dots,n-1$, then
$V(\bar{a})=V[x_{k_0},\dots,x_{k_{n-1}}]$ where $x_{k_0},\dots,x_{k_{n-1}}$ is generic for the poset adding $n$ many generic subsets of $W$.
The sequence $\seqq{x_j}{j\neq k_i,\, i<n}$ is generic over $V(\bar{a})$ for the poset adding a function $(\omega\setminus \{k_0,\dots,k_{n-1}\})\times W \to \{0,1\}$ (which is isomorphic to $\Q(W)$), and $V(A)=V(\bar{a})(B)$ where $B=\set{[x_j]}{j\neq k_i,\, i<n}=A\setminus\bar{a}$.
Furthermore, $T\restriction B$ is identified with a $\T(B,W)$-generic over $V(\bar{a})(B)$, and $V(A,T)=V(\bar{a})(B)(T\restriction B)$.
\end{remark}

\begin{lemma}\label{lemma: tree stepup no subset of V}
Forcing with $\T$ over $V(A)$ adds no new subsets to sets of $V$.
\end{lemma}
\begin{proof}
Let $T\subset\T$ be generic over $V(A)$ and suppose $X\in V(A, T)$ is a subset of $V$.
Assume first that $X$ is definable in $V(A, T)$ as the set of all solutions $\{x\in V \mid \phi(x, A,T,v)\}$.
By Lemma~\ref{lemma: basic tree step}, $x\in X$ if and only if $\T\force\phi(\check{x},A,\dot{T},v)$, in $V(A)$. So $X$ is in $V(A)$.

For an arbitrary $X\in V(A,T)$, there is some finite $\bar{a}\subset A$ and a parameter $v\in V(\bar{a})$ so that $X$ is defined, in $V(A)=V(\bar{a})(A\setminus\bar{a})$, as the set of all solutions $\phi(x,A,T,v)$ for $x\in V$.
From the previous argument it follows that $X\in V(\bar{a})(A\setminus\bar{a})=V(A)$.
\end{proof}

\section{The construction}\label{section:construction}
Begin with a model $V$ of $\ZF$, serving as the ground model for the construction. We will be particularly interested in a model $V$ satisfying $\ZFC$ and having many large cardinals.
We define recursively along the ordinals $\alpha$ sets $A_\alpha$, relations $T_\alpha$, and models $M_\alpha$ such that
\begin{itemize}
    \item $T_\alpha$ is a tree of height $\alpha$ whose $\beta$'th level is $A_\beta$;
    \item for $\alpha<\beta$, $T_\beta$ extends $T_\alpha$;
    \item Let $\theta$ be a limit ordinal. Define $A_{<\theta}=\bigcup_{\alpha<\theta}A_\alpha$, and let $T_{<\theta}$ be the tree on $A_{<\theta}$ defined as the union of $T_{\alpha}$ for $\alpha<\theta$. Then $A_\theta$ is a set of cofinal branches in $T_{<\theta}$;
    \item $M_\alpha=V(A_\alpha,T_\alpha)$.
\end{itemize}
For $\alpha<\beta$, the pair $(A_\alpha,T_\alpha)$ will be definable from $(A_\beta,T_\beta)$ and the ordinal $\alpha$, so $M_\beta$ is an extension of $M_\alpha$.
These sets will live in some generic extensions of $V$.

Let $A=\bigcup_{\alpha} A_\alpha$ and let $T$ be a tree on $A$ such that $A_\alpha$ is the $\alpha$-th level of $T$. Our final model will be $V(A, T)$. 
For $a \in A$, we will denote by $\proj_\beta(a)$ the unique element in $A_\beta$ below $a$, assuming that the level of $a$ is at least $\beta$. If the level of $a$ is below $\beta$, we set $\proj_\beta(a) = a$. 

Let us now give a formal definition of an iterated forcing such that $V(A, T)$ is an inner model of the generic extension. 

\begin{defn}[Definition of the construction]\label{definition: the iteration}
In the first stage, we define $A_1$ as follows.
\begin{itemize}
    \item $\Q_0=$ Cohen, adding $x_0\colon\omega\times\omega\to 2$;
    \item $x_0(n)=\set{m}{x_0(n,m)=1}$;
    \item $A_1(n)=[x_0(n)]=\set{y\subset\omega}{y\Delta x_0(n)\textrm{ is finite}}$;
    \item $A_1=\set{A_1(n)}{n\in\omega}$.
\end{itemize}
For the successor stages, suppose $A_\alpha$ is given, we construct $A_{\alpha+1}$ and the tree structure between $A_{\alpha}$ and $A_{\alpha+1}$.
\begin{itemize}
    \item $\Q_\alpha=\Q(A_\alpha)$ adds  $x_\alpha\colon\omega\times A_\alpha\to 2$ by finite conditions;
    \item $x_\alpha(n)=\set{a\in A_\alpha}{x_\alpha(n,a)=1}$;
    \item $A_{\alpha+1}(n)=[x_\alpha(n)]=\set{y\subset A_\alpha}{y\Delta x_\alpha(n)\textrm{ is finite}}$;
    \item $A_{\alpha+1}=\set{A_{\alpha+1}(n)}{n\in\omega}$.
    \item $\R_\alpha = \mathbb{T}(A_{\alpha+1},A_\alpha)$ adds $\pi_\alpha\colon A_{\alpha+1}\to A_\alpha$ by finite conditions;
    \item for $a\in A_\alpha$ and $b\in A_{\alpha+1}$, define $a<_Tb \iff \pi_\alpha(b)=a$. 
\end{itemize}
Given a limit ordinal $\lambda$, given a tree $\bigcup_{\alpha<\lambda} T_\alpha$ on $\bigcup_{\alpha<\lambda}A_\alpha$, we construct $A_\lambda$ as a set of branches as follows.
\begin{itemize}
    \item $\mathbb{B}_\lambda$ is the poset of finite functions from $\omega$ to the tree, where extension is defined by going up in the tree coordinate-wise or extending the domain.
    \item The generic filter of the forcing $\mathbb{B}_\lambda$ is essentially a sequence of branches $\seqq{b_n}{n<\omega}$;
    \item let $A_\lambda=\set{b_n}{n\in\omega}$, the unordered set of branches.
    \item We extend the tree order to $A_{\lambda}$ by defining $a <_T b$ for $a \in b \in A_\lambda$.
     
\end{itemize}
For a successor ordinal $\alpha$ we define $B_\alpha$ to be the trivial forcing.

Given $\Q_\alpha,\R_\alpha,\B_\alpha$ for $\alpha<\lambda$ as above define $\P_\lambda$ as the finite support iteration. Namely, $\mathbb{P}_0$ is the trivial forcing, $\mathbb{P}_{\alpha+1} = \mathbb{P}_\alpha \ast \mathbb{B}_\alpha\ast\mathbb{Q}_\alpha \ast \mathbb{R}_\alpha$ and for limit ordinal $\beta$, $\mathbb{P}_\beta$ is the direct limit of $\mathbb{P}_\alpha$ for $\alpha < \beta$.\footnote{For more information about iterated forcing, we refer the reader to \cite{Kunen2011}.} 
\end{defn}

The key property of the model is that the types of members of $A$ are determined by the tree structure.

\begin{defn}\label{defn: tree type}
Given $\Bar{a}=a_1,\dots,a_n$ in $A$, the \textbf{tree type} of $\Bar{a}$ is the structure $\left(n\times\mathrm{Ord},\approx,T,\seqq{L_\alpha}{\alpha\in\mathrm{Ord}}\right)$ defined by
\begin{itemize}
    \item $(i,\alpha)\approx (j,\beta)\iff \proj_\alpha(a_i)=\mathrm{proj}_{\beta}(a_j)$; 
    \item $(i,\alpha)\mathrel{T}(j,\beta)\iff \proj_\alpha(a_i)<_T \proj_\beta(a_j)$;
    \item $(i,\alpha)\in L_\alpha\iff \proj_\alpha(a_i)\in A_\alpha$.
\end{itemize}
\end{defn}
Let us remark that the tree type of $\Bar{a}$ is (essentially) a finite object, coding the levels of the elements of $\Bar{a}$, the equality relation on them, and the structure of the finite tree $<_T \restriction \Bar{a}$, including the level of their meets. 

Consider $M_\alpha=V(A_\alpha,T_\alpha)$. Suppose $\bar{a},\bar{b}$ are from $\bigcup_{\beta>\alpha}A_\beta$.
Say that $\bar{a},\bar{b}$ \textbf{have the same tree type over $A_\alpha$} if $\bar{a}^\frown\bar{u}$ and $\bar{b}^\frown\bar{u}$ have the same tree type, for any finite $\bar{u}\subset A_\alpha$.
For example, if $x,y$ are any two elements in the same level $A_\theta$, then $x$ and $y$ have the same tree type.
They have the same tree type over $A_\alpha$ if moreover, their projections to level $\alpha$ are the same. 

More generally, if $\bar a, \bar b$ are pairs of elements in $\bigcup_{\beta > \alpha} A_\beta$, then they have the same tree type over $A_\alpha$ if and only if $\bar{a}\smallfrown \bar{a'}$ and $\bar b \smallfrown \bar b'$ have the same tree type, where $\bar a'(k) = \proj_\alpha (\bar a(k))$, $\bar b'(k) = \proj_\alpha (\bar b(k))$ for all $k < \len \bar a$ or $\len \bar b$ respectively.

\begin{defn}\label{defn: tree indiscernibility}
The \textbf{tree indiscernibility hypothesis} is the following statement:
For any ordinal $\alpha$, parameter $v\in M_\alpha$, and formula $\phi$, suppose that $\bar{a},\bar{b}$ are from $\bigcup_{\beta>\alpha}A_\beta$ and have the same tree type over $A_\alpha$, then
\begin{equation*}
    \phi(A,T,\bar{a},v)\iff \phi(A,T,\bar{b},v).
\end{equation*}
\end{defn}
Our goal will be to prove the tree indiscernibility hypothesis inductively along our construction. In the process we will consider the hypothesis in models of the form $M_\theta=V(A_\theta, T_\theta)$, or $V(A_{<\theta}, T_{<\theta})$, in which case $A, T$ above are replaced with $A_\theta, T_\theta$ or $A_{<\theta}, T_{<\theta}$ respectively.

For $\theta'<\theta$, the formula $\phi^{M_{\theta'}}(A_{\theta'},T_{\theta'},\bar{a},v)$ can be expressed in the model $M_\theta$ as $\varphi(A_\theta,T_\theta,\bar{a},v,\theta')$ for some formula $\varphi$. It follows that the tree indiscernibility hypothesis at $M_\theta$ implies the tree indiscernibility hypothesis at $M_{\theta'}$ for $\theta'<\theta$.

\begin{lemma}[Propagation of tree indiscernibility]\label{lemma: indiscernibility propagation}
Using the notations $M_{\theta}=V(A_\theta,T_\theta)$ and the notion of tree indiscernibility from above: 
\begin{enumerate}
    \item {\bf Successor step:} Assume tree indiscernibility in $M_\theta$. Then tree indiscernibility holds in $M_{\theta+1}$.
    \item {\bf Limit step:} Let $\theta$ be a limit ordinal. Suppose the tree indiscernibility hypothesis holds at $M_\alpha$ for each $\alpha<\theta$. Then the tree indiscernibility hypothesis holds in the model $M_{<\theta} = V(A_{<\theta},T_{<\theta})$.
    \item {\bf Tree step:} Suppose the tree indiscernibility hypothesis holds in the model $M_{<\theta}$. Then it holds in $M_\theta$.
\end{enumerate}
\end{lemma}
We first focus on a step-by-step approach to the construction. We prove part (3) in
Section~\ref{subsection: tree step}. This is the main point where our tree structure is being used: to allow us to add sets of limit rank (branches through the tree) without adding sets of lower rank, and in fact while preserving the indiscernibility.
Part (1) is proved in Section~\ref{section: basic lemmas for successor step}.
Finally, we prove part (2) of the lemma in Section~\ref{section:construction}. 

Before proving Lemma~\ref{lemma: indiscernibility propagation}, let us collect important properties of the final model $V(A, T)$ which follow from it. We will also show inductively that sets of bounded rank are stabilized through the construction.
\begin{prop}\label{prop: stabilization of initial ranks}
For any ordinals $\alpha<\beta$, if $X\in M_\beta$ and $X\subset M_\alpha$, then $X\in M_{\alpha+1}$. In particular, for $\alpha<\beta$, $M_\alpha$ and $M_\beta$ agree on $\mathcal{P}^\alpha(\mathrm{Ord})$.
\end{prop}
This implies that the Kinna-Wagner degree is increasing. Recall that two transitive models with Kinna-Wagner degree $\alpha$ which agree on $\mathcal{P}^{\alpha+1}(\Ord)$ are identical. For $\alpha = 0$, this is a theorem of Balcar and Vopenka: two transitive ZFC models which agree on sets of ordinals are identical. The generalization is stated for $\alpha<\omega$ in \cite{Monro-1973} and for all ordinals in \cite[Theorem~10.3]{Karagila-iterating-2019}.
\begin{cor}
    The Kinna-Wagner degree of the models $M_\alpha$ is unbounded.
\end{cor}
\begin{proof}
    Otherwise, we could find $\alpha<\beta$ where both $M_\alpha, M_\beta$ have Kinna-Wagner degree $\gamma$ for some $\gamma < \alpha$. Since $M_\alpha, M_\beta$ agree on $\mathcal{P}^{\gamma}(\mathrm{Ord})$, we conclude, by the generalized Balcar-Vopenka theorem, that $M_\alpha = M_\beta$, a contradiction, as $M_{\alpha+1}\subset M_\beta$, and $M_{\alpha+1}$ contains sets which are generic over $M_\alpha$.
\end{proof}
\begin{cor}\label{cor: unbdd KW for V(A, T)}
    The Kinna-Wagner degree of $V(A, T)$ is $\infty$.
\end{cor}
\begin{proof}
$V(A,T)$ is defined as the union of $M_\alpha$ over all ordinals $\alpha$. For any ordinal $\alpha$, we may find some $\alpha < \beta$ and a set $X\in M_\beta$ so that in $M_\beta$ there is no injective map from $X$ to $\mathcal{P}^{\alpha}(\mathrm{Ord})$. We may assume that $X$ is transitive. Since $V(A,T)$ and $M_\beta$ agree on $\mathcal{P}^{\alpha+1}(\mathrm{Ord})$, using Remark~\ref{remark: coding injective map} we conclude that there is no injective map from $X$ to $\mathcal{P}^\alpha(\mathrm{Ord})$ in $V(A,T)$.
\end{proof}

\begin{cor}\label{cor: AC fails in set forcing}
    The axiom of choice cannot be forced by a set forcing over $V(A, T)$.
\end{cor}
\begin{proof}
This follows by Corollary~\ref{cor: unbdd KW for V(A, T)} and Lemma~\ref{lemma: bound on KW from forcing AC}.
\end{proof}
In Section~\ref{subsection: full indiscernibility} we will prove an extension of the tree indiscernibility, which will provide more precise statements about fragments of choice. For example, we will see that DC fails in any set-forcing extension of $V(A, T)$.

We proceed with the proof of Lemma \ref{lemma: indiscernibility propagation}. We will change the order of the proof, and prove the different cases according to their difficulty. We start by proving part (3), then part (1), and lastly part (2).

\subsection{Tree step}\label{subsection: tree step}
Let $\theta$ be a limit ordinal.
Assume $M_{<\theta}$ was constructed, and further that the tree indiscernibility hypothesis holds.
We now describe the construction of $M_\theta$, and prove part (3), the Tree step of Lemma~\ref{lemma: indiscernibility propagation}.

Let $\B=\B_\theta$ be the poset of all finite partial functions $p$ from $\omega$ to $A_{<\theta}$. Say that a condition $p$ extends $q$ if the domain of $p$ extends the domain of $q$ and $p(i)$ is above $q(i)$ in the tree for any $i$ in the domain of $q$.
Given a natural number $k$, let $\B\restriction k$ be the poset of all conditions $p\in\B$ whose domain is contained in $k=\{0,\dots,k-1\}$.
Given $p\in\B\restriction k$ and $\alpha<\theta$, let $p\restriction \alpha$ be the condition $p$ restricted to the $\alpha$'th level of the tree. That is, for $i<k$, $(p\restriction \alpha)(i)$ is the unique element in $A_\alpha$ below $p(i)$ in the tree if $p(i)$ is in a higher level.
Let $\dot{b}_i$ be the canonical name for the $i$'th generic branch, $\set{p(i)}{p\in\dot{G}}$.
\begin{claim}
Suppose $v\in M_\alpha$, $\psi$ a formula such that \[p\force_{\B\restriction k}\psi(A_{<\theta},T_{<\theta},v,\dot{b}_0,\dots,\dot{b}_{k-1}).\]
Assume further that the projections of $p(0),\dots,p(k-1)$ to level $\alpha$ of the tree are distinct.
Then 
\begin{equation*}
    (p\restriction\alpha)\force_{\B\restriction k}\psi(A_{<\theta},T_{<\theta},v,\dot{b}_0,\dots,\dot{b}_{k-1}).
\end{equation*}
\end{claim}
\begin{proof}
We may extend $p$ and assume that $p(0),\dots,p(k-1)$ are in the same level $A_\beta$.
We show that if $q$ is an extension of $p\restriction\alpha$ such that $q(0),\dots,q(k-1)$ are in $A_\beta$, then $q \force_{\mathbb{B}\restriction k} \psi(A_{<\theta},T_{<\theta},v,\dot{b}_0,\dots,\dot{b}_{k-1})$. This suffices as such conditions $q$ are predense in $\B\restriction k$ below $p\restriction\alpha$.

For any such $q$, $p$ and $q$ have the same tree type over $A_\alpha$.
Note that the statement ``$p\force_{\B\restriction k}\psi(A_{<\theta},T_{<\theta},v,\dot{b}_0,\dots,\dot{b}_{k-1})$'' can be expressed, in $V(A_{<\theta},T_{<\theta})$, as $\phi(p,A_{<\theta},T_{<\theta},v)$ for some formula $\phi$.
By the inductive hypothesis, we conclude that
\begin{equation*}
    \phi(p,A_{<\theta},T_{<\theta},v)\iff \phi(q,A_{<\theta},T_{<\theta},v),
\end{equation*}
and therefore $q\force_{\B\restriction k}\psi(A_{<\theta},T_{<\theta},v,\dot{b}_0,\dots,\dot{b}_{k-1})$.
\end{proof}

\begin{proof}[Proof of Lemma~\ref{lemma: indiscernibility propagation} part (3) {[Tree step]}]

Let $\seqq{b_i}{i<\omega}$ be a $\B$-generic such that $A_\theta=\set{b_i}{i\in\omega}$. Recall that $T_\theta$ is defined by declaring, for $a\in A_{<\theta}$, $a<_T b_i$ if $a\in b_i$, so that $T_\theta$ is definable from $T_{<\theta},A_{<\theta},A_\theta$.

The forcing $\B$ is isomorphic to the finite support product of $\omega$ many copies of $\B\restriction 1$, where $\B\restriction 1$ is the poset for adding a single branch through $T_{<\theta}$.
In particular, for any tuple $\bar{b}=\langle b_{k_0},\dots,b_{k_{n-1}}\rangle$ of distinct elements from $A_\theta$, $\bar{b}$ is $\B\restriction k$-generic over $V(A_{<\theta},T_{<\theta})$.
Note that the poset $\B$ is definable from $A_{<\theta}$ and $T_{<\theta}$.
Applying Lemma~\ref{lemma: basic permutation}, for any formula $\phi$ and parameter $v\in V(A_{<\theta},T_{<\theta})$ there is a formula $\psi(A_{<\theta},T_{<\theta},\bar{b},v)$ such that
\begin{equation*}
    V(A_\theta,T_\theta)\models\phi(A_\theta,T_\theta,\bar{b},v)\iff V(A_{<\theta},T_{<\theta})[\bar{b}]\models \psi(A_{<\theta},T_{<\theta},\bar{b},v).
\end{equation*}

We now establish the tree indiscernibility hypothesis in $V(A_\theta, T_\theta)$. Fix $v\in M_\alpha$, $\alpha<\theta$, a formula $\phi$, $\bar{a}_0,\bar{b}_0$ from $\bigcup_{\theta>\beta>\alpha}A_\beta$ and $\bar{a},\bar{b}$ from $A_\theta$ such that $\bar{a}_0,\bar{a}$ and $\bar{b}_0,\bar{b}$ have the same tree type over $A_\alpha$. Assume further that $\phi(A_\theta,T_\theta,\bar{a}_0,\bar{a},v)$ holds, and show that $\phi(A_\theta,T_\theta,\bar{b}_0,\bar{b},v)$ holds as well.

Fix $\psi(A_{<\theta},T_{<\theta},\bar{a}_0,\bar{a},v)$ as above.
Fix an ordinal $\beta < \theta$ such that the elements of $\bar{a}_0$ are below level $\beta$, and such that the projections of $\bar{a}$ to level $\beta$ are all distinct.
Let $p=\seq{p(0),\dots,p(k-1)}$ and $q=\seq{q(0),\dots,q(k-1)}$ be the projections of $\bar{a}$ and $\bar{b}$ to level $\beta$, respectively.
Since $\bar{a}_0,\bar{a}$ and $\bar{b}_0,\bar{b}$ have the same tree type, it follows that $\bar{b}_0$ are below level $\beta$, that the projections of $\bar{b}$ to level $\beta$, $q(0),\dots,q(k-1)$, are distinct, and that $\bar{a}_0,p$ and $\bar{b}_0,q$ have the same tree type over $A_\alpha$.

By assumption, $\phi(A_\theta,T_\theta,\bar{a}_0,\bar{a},v)$ holds in $V(A_\theta,T_\theta)$. Therefore $V(A_{<\theta},T_{<\theta})[\bar{a}]$ satisfies $\psi(A_{<\theta},T_{<\theta},\bar{a}_0,\bar{a},v)$. There is some condition $\tilde{p}\in \mathbb{B}\restriction k$ so that $\tilde{p}\force_{\mathbb{B}\restriction k}\psi(A_{<\theta},T_{<\theta},\bar{a}_0,\dot{b}_0,\dots,\dot{b}_{k-1},v)$ and $\tilde{p}(i)$ is below $a_i$, that is, $\tilde{p}$ is in the $\mathbb{B}\restriction k$-generic corresponding to $\bar{a}$. We may assume that $\tilde{p}(i)$ is above level $\alpha$, for each $i$, and so $p$ is the restriction of $\tilde{p}$ to level $\alpha$. By the claim above, we conclude in $V(A_{<\theta},T_{<\theta})$ that
\begin{equation*}
    p\force_{\B\restriction k}\psi(A_{<\theta},T_{<\theta},\bar{a}_0,\dot{b}_0,\dots,\dot{b}_{k-1},v).
\end{equation*}
This displayed statement can be written in $V(A_{<\theta},T_{<\theta})$ as $\chi(p,A_{<\theta},T_{<\theta},\bar{a}_0,v)$.
Finally, as $\bar{a}_0,p$ and $\bar{b}_0,q$ have the same tree type over $A_\alpha$, using the inductive tree indiscernibility hypothesis, we conclude that $\chi(q,A_{<\theta},T_{<\theta},\bar{b}_0,v)$ holds in $V(A_{<\theta},T_{<\theta})$, which in turn implies that $q\force_{\B\restriction k}\psi(A_{<\theta},T_{<\theta},\bar{b}_0,\dot{b}_0,\dots,\dot{b}_{k-1},v)$. Therefore $\psi(A_\theta,T_\theta,\bar{b}_0,\bar{b},v)$ holds in $V(A_{<\theta},T_{<\theta})[\bar{b}]$, and so $\phi(A_\theta,T_\theta,\bar{b}_0,\bar{b},v)$ holds in $V(A_\theta,T_\theta)$, as required.
\end{proof}

\subsection{Successor step}\label{section: basic lemmas for successor step}

Assume that $M_\theta$ was constructed, and the tree indiscernibility hypothesis holds.
Fix a $\Q(A_\theta)\ast \mathbb{T}(\dot{A}_{\theta+1},A_\theta)$-generic over $M_\theta$, let $A_{\theta+1}\subset\mathcal{P}(A_\theta)$ and $\pi\colon A_{\theta+1}\to A_\theta$ be the corresponding generic objects, and $T_{\theta+1}$ defined as the extension of $T_{\theta}$ using $\pi$.

\begin{proof}[Proof of Lemma~\ref{lemma: indiscernibility propagation} part (1) {[Successor step]}]
Fix an ordinal $\alpha<\theta+1$, $v\in M_\alpha$, a formula $\phi$, tuples $\bar{a},\bar{b}$ from $\bigcup_{\theta+1\geq\beta>\alpha}A_\beta$ which have the same tree type over $A_\alpha$, and assume that $\phi(A_{\theta+1},T_{\theta+1},\bar{a},v)$ holds in $M_{\theta+1}$. We need to show that $\phi(A_{\theta+1},T_{\theta+1},\bar{b},v)$ holds as well.

Assume first that $\alpha=\theta$. Then $\bar{a},\bar{b}$ are sequences from $A_{\theta+1}$. Since they have the same type over $A_\theta$, it follows that the projection of $a_i$ to level $\theta$ is equal to the projection of $b_i$, for each $i$.
By Lemma~\ref{lemma: basic tree step}, we conclude that
\begin{equation*}
    \phi(A_{\theta+1},T_{\theta+1},\bar{a},v)\iff\phi(A_{\theta+1},T_{\theta+1},\bar{b},v).
\end{equation*}
Assume now $\alpha<\theta$. Let $\bar{a}=\bar{a}_0,\bar{a}_1$ and $\bar{b}=\bar{b}_0,\bar{b}_1$, where $\bar{a}_0,\bar{b}_0$ are from $\bigcup_{\theta\geq\beta>\alpha}$.
Define $\bar{w},\bar{u}$ to be the projections of $\bar{a}_1,\bar{b}_1$, respectively, to level $\theta$. Since $\bar{a}_0,\bar{a}_1$ and $\bar{b}_0,\bar{b}_1$ have the same tree type over $\alpha$, it follows that $\bar{a}_0,\bar{w}$ and $\bar{b}_0,\bar{u}$ have the same tree type over $\alpha$.
Furthermore, by Lemma~\ref{lemma: basic tree step}, there is a formula $\psi$ such that
\begin{equation*}
    M_\theta\models\psi(A_\theta,T_\theta,\bar{a}_0,\bar{w},v)\iff M_{\theta+1}\models\phi(A_{\theta+1},T_{\theta+1},\bar{a}_0,\bar{a}_1,v),
\end{equation*}
whenever $\bar{w}$ are the projections of $\bar{a}_1$ to $A_\theta$.
(The lemma is applied with $W=A_\theta$ and $T_\theta,\bar{a}_0,v$ as the parameters from the ground model $M_\theta$.)
Applying the inductive tree indiscernibility hypothesis, $ M_\theta\models\psi(A_\theta,T_\theta,\bar{b}_0,\bar{u},v)$ as well, and therefore $M_{\theta+1}\models\phi(A_{\theta+1},T_{\theta+1},\bar{b}_0,\bar{b}_1,v)$, as required.
\end{proof}

\subsection{The limit step}\label{subsection: limit-step}
In order to prove Lemma~\ref{lemma: indiscernibility propagation} part (2) [Limit step], we will need to construct automorphisms for quotients of the iteration $\mathbb{P}_\theta$ for some limit ordinal $\theta$. 
For this proof, the following setup is required.

Let $\P$ be a class iteration and $\P$-names $\dot{A}_\alpha,\dot{T}_\alpha$, for all ordinals $\alpha$.
Given a generic $G$, let  $A_\alpha$ and $T_\alpha$ be the interpretations of $\dot{A}_\alpha,\dot{T}_\alpha$ according to $G$, and define $M_\alpha=V(A_\alpha,T_\alpha)$.
Assume that
\begin{itemize}
    \item for each ordinal $\alpha$ there is a $\Q_\alpha\ast\R_\alpha$-generic over $M_\alpha$ (in some generic extension), as in Section~\ref{section: basic lemmas for successor step}, such that $M_{\alpha+1}=M_\alpha(A_{\alpha+1},T_{\alpha+1})$ is the resulting symmetric model as in Section~\ref{section: basic lemmas for successor step};
    \item for each limit ordinal $\theta$, let $A_{<\theta}=\bigcup_{\alpha<\theta}A_\alpha$ and $T_{<\theta}=\bigcup_{\alpha<\theta}T_\alpha$, then there is a $\B_\alpha$-generic over $V(A_{<\theta}, T_{<\theta})$ (in some generic extension), as in Section~\ref{subsection: tree step}, such that $A_{\theta}$ is the unordered set of branches added by this generic.
\end{itemize}
Furthermore, assume there are posets $\P_\alpha^\beta\in V(A_\alpha,T_\alpha)$, with names $\dot{A}_\alpha^\beta, \dot{T}_\alpha^\beta$ such that
\begin{enumerate}
    \item\label{property:constructed from generic} There is a generic $G_\alpha^\beta$ for $\P_\alpha^\beta$ over $V(A_\alpha,T_\alpha)$ for which $\dot{A}_\alpha^\beta[G_\alpha^\beta]=A_\beta$, $\dot{T}_\alpha^\beta[G_\alpha^\beta]=T_\beta$;
    \item\label{property: rich automorhpisms} For any generic $G_\alpha^\beta$ and any condition $q\in \P_\alpha^\beta$ there is a generic $\Tilde{G}_\alpha^\beta$ with
    \begin{itemize}
        \item $q\in \Tilde{G}_\alpha^\beta$;
        \item $\dot{A}_\alpha^\beta[G_\alpha^\beta]=\dot{A}_\alpha^\beta[\tilde{G}_\alpha^\beta]$ and $\dot{T}_\alpha^\beta[G_\alpha^\beta]=\dot{T}_\alpha^\beta[\tilde{G}_\alpha^\beta]$.
    \end{itemize}
\end{enumerate}
Properties (1) and (2) above say that stage $\beta$ can be forced over each $M_\alpha$ for $\alpha<\beta$, in a sufficiently homogeneous way.
The definition of the iteration was given in Definition~\ref{definition: the iteration}.
Properties (\ref{property:constructed from generic}) and (\ref{property: rich automorhpisms}) are proven below in Sections \ref{subsubsection: quotients} and \ref{subsubsection: permutations} respectively.
First, let us see how these properties allow us to push the construction through the limit stages.
\begin{proof}[Proof of Lemma~\ref{lemma: indiscernibility propagation} part (2) {[Limit step]}]
Assume the tree indiscernibility hypothesis in $M_\alpha=V(A_\alpha,T_\alpha)$ for every $\alpha<\theta$.
Fix $\alpha<\theta$, $v\in M_\alpha$, a formula $\phi$ and tuples $\bar{a},\bar{b}$ from $\bigcup_{\theta>\beta>\alpha} A_\beta$ which have the same tree type over $A_\alpha$.
Assume that $\phi(A_{<\theta},T_{<\theta},\bar{a},v)$ holds in  $V(A_{<\theta},T_{<\theta})$, and show that $\phi(A_{<\theta},T_{<\theta},\bar{b},v)$ holds as well.

Fix $\beta<\theta$ large enough such that $\bar{a}$ is below level $\beta$. It follows that $\bar{b}$ is also below level $\beta$, and $v\in M_\beta$.
By the homogeneity property (\ref{property: rich automorhpisms}) above, 
\begin{equation*}
    \phi^{V(A_{<\theta},T_{<\theta})}(A_{<\theta},T_{<\theta},\bar{a},v)\iff V(A_\beta,T_\beta)\models \P_\beta^\theta\force\phi^{V(\dot{A}_{<\theta},\dot{T}_{<\theta})}(\dot{A}_{<\theta},\dot{T}_{<\theta},\bar{a},v).
\end{equation*}
The latter statement can be expressed in $V(A_\beta,T_\beta)$ as $\psi(A_\beta,T_\beta,\bar{a},v)$ for some formula $\psi$.
By the inductive tree indiscernibility hypothesis, we conclude that $\psi(A_\beta,T_\beta,\bar{b},v)$ holds in $V(A_\beta,T_\beta)$ as well, that is,
\begin{equation*}
    V(A_\beta,T_\beta)\models \P_\beta^\theta\force\phi^{V(\dot{A}_{<\theta},\dot{T}_{<\theta})}(\dot{A}_{<\theta},\dot{T}_{<\theta},\bar{b},v),
\end{equation*}
and therefore $\phi(A_{<\theta},T_{<\theta},\bar{b},v)$ holds in $V(A_{<\theta},T_{<\theta})$.
\end{proof}

\subsubsection{Quotients}\label{subsubsection: quotients}
As in Definition~\ref{definition: the iteration}, define $\P_\alpha^\beta$ as the iteration, over $V(A_\alpha,T_\alpha)$, starting with $\Q_\alpha$ and ending with $\B_\beta$. Property (\ref{property:constructed from generic}) now follows immediately.

\subsubsection{Permutations}\label{subsubsection: permutations}
We describe below some permutations of the iteration $\P$ which witness that $\P$ is weakly homogeneous while fixing $\dot{A}_\alpha$ and $\dot{T}_\alpha$ for all ordinals $\alpha$. These are used to establish clause (\ref{property: rich automorhpisms}) above, and will also be used later in Section \ref{section: construction in Cohen extension}. The permutations are defined as compositions of finitely many permutations, each of which deals with a single coordinate in the support of the given condition. 

We begin with coordinates which are successor ordinals. For any sequence $t=\bar{a}_0,\dots,\bar{a}_m$ with $\bar{a}_i$ a finite (possible empty) subset of $A_\alpha$, define an automorphism $f^t=f^\alpha_t$ of $\Q_\alpha$ by flipping the values of $p(i,a)$ for $a\in \bar{a}_i$. That is, for $a\in \bar{a}_i$, if $(i,a)\in\dom p$ then
\begin{equation*}
    f_t(p)(i,a)=1-p(i,a).
\end{equation*}
Note that $f_t$ fixes $\dot{A}_{\alpha+1}(n)$ for each $n$, as these are defined as the equivalence classes of all finite changes. In particular, $\dot{A}_{\alpha+1}$ is fixed, as well as the tree structure introduced by $\mathbb{R}_\alpha$.

Fix a permutation $\sigma$ of $\omega$ with finite support.
Define an automorphism $a_\sigma$ of $\R_\alpha$ by 
\begin{equation*}
    a_\sigma(p)=p\circ\sigma.
\end{equation*}
Define an automorphism $e_\sigma$ of $\dot{\Q}_\alpha$ by
\begin{equation*}
    e_\sigma(p)(\sigma(m),-)=p(m,-).
\end{equation*}
Applying $a_\sigma$ changes the tree structure between $A_{\alpha+1}$ and $A_\alpha$, but an application of $e_\sigma$ corrects that.
Then applying $e_\sigma\circ a_\sigma$ to $\Q_\alpha\ast\R_\alpha$ fixes $\dot{A}_\gamma$ and the tree structure for all $\gamma$. Let 
\begin{equation*}
    g^\alpha_\sigma=e_\sigma\circ a_\sigma.
\end{equation*}

Given any generic $G$ for $\Q_\alpha\ast\R_\alpha$ and any condition $q$ in $\Q_\alpha\ast\R_\alpha$, there is some $t$ and $\sigma$ as above such that $f^\alpha_t\circ g^\alpha_\sigma(q)\in G$.
We first choose $\sigma$ to make the $\R_\alpha$ coordinate of $q$ agree with $G$, and then choose $t$ to change the values of $A_{\alpha+1}(n)$ defined by $g^\alpha_\sigma(q)$ to agree with $G$. This is possible as $q$ has finite support.

Let us deal now with the tree step. In this case, we need to construct an automorphism of $\mathbb{B}_{\alpha}$. Let $\sigma$ be a finite support permutation of $\omega$. Define $b^\alpha_\sigma$ by
\begin{equation*}
    b^\alpha_\sigma(p) = p\circ \sigma
\end{equation*}
for $p \in \mathbb{B}_{\alpha}$. Given a condition $q \in \B_\alpha$, a density argument shows that there is $\sigma$ such that $b^\alpha_{\sigma}(q)$ is compatible with $G$. Note that the set of all $p\in \B_\alpha$ so that $p$ is compatible with $b^\alpha_\sigma(q)$ for \emph{some} $\sigma$, is dense in $\B_\alpha$. This is true because we can take $\sigma$ so that the domains of $b^\alpha_\sigma(q)$ and $p$ are disjoint, and therefore their union is a well-defined extension of $p$.
Note that $b^\alpha_\sigma$ fixes $\dot{A}_\alpha$, as the elements there are determined up to a finite modification.  

When considering $f^\alpha_t$, $g^\alpha_\sigma$ or $b^\alpha_\sigma$ above as permutations of the iterations $\P_\gamma$ or $\P_\gamma^\beta$, for $\gamma<\alpha<\beta$, they fix everything outside $\Q_\alpha, \R_\alpha, \B_\alpha$.
Property (\ref{property: rich automorhpisms}) above now follows: given a generic $G$ for the iteration and a condition $p$, by applying the above corrections, finitely many times along the support of $p$, we can move $p$ into $G$.

\begin{defn}\label{definition:group-of-automorphisms-for-P_alpha}
Let $\mathcal{G}_\alpha$ be the group of automorphisms of $\mathbb{P}_\alpha$ generated by the automorhpisms $f^\beta_t, g^\beta_\sigma, b^\beta_\sigma$ for $\beta < \alpha$, $t$ a finite sequence of finite subsets of $A_\beta$, and $\sigma$ a finite support permutation of $\omega$.
\end{defn}


\section{Full types and indiscernibility}\label{subsection: full indiscernibility}
The members of $A=\bigcup_{\alpha}A_\alpha$ and the tree structure $T$ were central to our work above. An arbitrary set in our model $V(A, T)$ is definable using $A$, $T$, members of $A_\alpha$ for some $\alpha$ (the vertices of our tree), but also from some sets which are members of the elements in $A_\alpha$.
Recall that each element $a\in A_{\alpha+1}$ is an equivalence class, of all finite changes, of some set $x\subset A_\alpha$. 
In this case, $a$ is definable from $x$ and $A_\alpha$, however, $x$ itself is not definable from $a$ and $A_\alpha$.
Next, we extend the indiscernibility to also consider sequences involving such sets.

The results of this section are not strictly necessary for the proof of our main result, Theorem~\ref{thm: intermediate extension}. Nevertheless, these results will be useful for future applications of our construction, so it is important to record them below. Furthermore, these results will be used to prove more precise statements about how fragments of choice fail in set-forcing extensions of $V(A, T)$, Proposition~\ref{prop: A-theta-D-finite-in-extensions} and Corollary~\ref{cor: failure of DC in set forcing}.

Fix some $a\in A_{\alpha+1}$. We do not have indiscernibility for pairs of members of $a$. Specifically, for any $x,y\in a$, the symmetric difference $x\Delta y$ is a finite subset of $A_\alpha$. The formula ``$|x\Delta y|=1$'' will hold for many pairs $x,y\in a$ and fail for others.
Instead, we will have indiscernibility as long as our parameters include at most one member out of each such $a$. In this case, other than the information in the tree type from Definition~\ref{defn: tree type}, we need to also include the relationship $x\in a$ in the types.

Say that a sequence $\bar{x}=\seq{x_1,\dots,x_n}$ is a sequence of \textbf{representatives} if there is a sequence $a_1,\dots,a_n$ of distinct sets from the tree $A$, such that $x_i\in a_i$. In other words, each $x_i$ is a member of a member of some $A_\alpha$, and $x_i,x_j$ are not equal mod finite, for any $i\neq j$.
\begin{obs}
For any set $S\in V(A,T)$  there is a formula $\varphi$, a parameter $v\in V$, a sequence $\bar{a}$ from $A$ and a sequence of representatives $\bar{x}$ such that in $V(A,T)$, $S$ is defined as the unique solution to $\varphi(S,A,T,\bar{a},\bar{x},v)$.
Equivalently, there is a formula $\psi$ such that in $V(A,T)$, $s\in S\iff \psi(s,A,T,\bar{a},\bar{x},v)$.
\end{obs}
Recall that the type of $\bar{a}$ was determined by its tree type. For pairs $\bar{a},\bar{x}$, we need to also consider the membership relation between the elements in $\bar{x}$ and the elements in the tree.
\begin{defn}\label{defn: full type}
Given $\Bar{a}=a_1,\dots,a_n$ and $\bar{x}=x_1,\dots,x_m$, the \textbf{full type} of $\bar{a},\Bar{x}$ is the structure $\left((n\times\mathrm{Ord})\sqcup \{s_0,\dots,s_{m-1}\},\approx,T,\seqq{L_\alpha}{\alpha\in\mathrm{Ord}},E\right)$ defined by
\begin{itemize}
    \item $(i,\alpha)\approx (j,\beta)\iff \mathrm{proj}_\alpha(a_i)=\mathrm{proj}_{\beta}(a_j)$; 
    \item $s_j\approx s_k\iff x_j=x_k$
    \item $(i,\alpha)\mathrel{T}(j,\beta)\iff \mathrm{proj}_\alpha(a_i)<_T \mathrm{proj}_\beta(a_j)$;
    \item $s_i\mathrel{E}(j,\beta)\iff x_i\in \mathrm{proj}_\beta(a_j)$, for $i<m$, $j<n$;
    \item $(i,\alpha)\mathrel{E}s_j\iff \mathrm{proj}_\alpha(a_i)\in x_j$, for $j<m$, $i<n$.
\end{itemize}
\end{defn}


For a sequence $\bar{u}$, containing both members of the tree and representatives, we will talk about the type of $\bar{u}$ in a natural way. Say that $\bar{a},\bar{x}$ and $\bar{b},\bar{y}$ \textbf{have the same type over $\bar u$} if $\bar{a},\bar{x}^\frown\bar{u}$ and $\bar{b},\bar{y}^\frown\bar{u}$ have the same full type.
Say that $\bar{a},\bar{x}$ and $\bar{b},\bar{y}$ have the same type over $A_\alpha$ if they have the same type over $\bar{u}$, for any finite $\bar{u}\subset A_\alpha$.

\begin{thm}\label{thm: full indiscernibility}
For any ordinal $\alpha$, parameter $v\in M_\alpha$, and formula $\phi$, suppose that $\bar{a},\bar{x}$ and $\bar{b},\bar{y}$ have the same type over $A_\alpha$, where $\bar{x}$ and $\bar{y}$ are sequences of representatives. Then
\begin{equation*}
    V(A,T)\models \phi(A,T,\bar{a},\bar{x},v)\iff \phi(A,T,\bar{b},\bar{y},v).
\end{equation*}
\end{thm}

We begin by expanding the results in Section~\ref{section: basic lemmas for successor step} to this context.
As in the notation of Section~\ref{section: basic lemmas for successor step}, let $V$ be some model of $\ZF$ and $W$ a set in $V$. The following is a seemly stronger bus equivalent form of Lemma~\ref{lemma: basic tree step}. In this setting, we consider $A$ to be the set of infinitely many equivalence classes of generic subsets of $W$, modulo finite modifications, and $\pi$ is a generic for $\mathbb{T}(A, W)$ - a generic surjection from $A$ to $W$.

In the notation of the lemma, suppose $\bar{a}$ is a finite sequence from $A$ and $\bar{y}$ such that $y_i\in a_i$ (so $a_i=[y_i]$).
\begin{lemma}\label{lemma: stronger basic tree step}
Fix a formula $\phi$ and parameter $v\in V$.
Suppose $\bar{b},\bar{c}$ are sequences from $A$ of length $n$ such that
\begin{itemize}
    \item $\bar{b},\bar{c}$ are disjoint from $\bar{a}$;
    \item $\pi(a_i)=\pi(b_i)$ for all $i<n$, and
    \item $a_i=a_j\iff b_i=b_j$ for any $i,j<n$, then
\end{itemize}
\begin{equation*}
    \phi(A,T,\Bar{y},\bar{b},v)\iff\phi(A,T,\Bar{y},\bar{c},v).
\end{equation*}
Moreover, for any formula $\phi(A,\Bar{y},\bar{b},v)$ there is a formula $\psi(W,\bar{y},\bar{w},v)$ such that for any $\bar{b}$, if $w_i=\pi(b_i)$ then
\begin{equation*}
    V[\bar{y}]\models\psi(W,\bar{y},\bar{w},v)\iff V(A,T)\models\phi(A,T,\bar{y},\Bar{b},v).
\end{equation*}
\end{lemma}
The proof above can be repeated to give this stronger statement. Instead, we deduce the lemma from the former, as follows.
\begin{proof}
Without loss of generality, assume that $\bar{y}=x_0,\dots,x_{m}$. Now $\seqq{x_i}{i> m}$ is $\Q(W)$-generic over $V[x_0,\dots,x_m]$, so the construction of $V(A)$ above can be instead presented with  $V[x_0,\dots,x_m]$ as the ground model, and $V(A)=V[x_0,\dots,x_m](\tilde{A})$, where $\tilde{A}=A\setminus\{x_0,\dots,x_m\}$. Note that $\tilde{A}$ and $A$ are definable from one another using $x_0,\dots,x_m$, which are now considered parameters in the ground model.
Let $\tilde{\pi}\colon \tilde{A}\to W$ be the restriction of $\pi$ to $\tilde{A}$. Then $\tilde{\pi}$ is generic for the poset $\tilde{\T}$, defined as $\T$ above using $\tilde{A}$. Again, $\tilde{\pi}$ and $\pi$ are definable from one another using $x_0,\dots,x_m$.
The conclusion now follows from Lemma~\ref{lemma: basic tree step}, as $\bar{x}$ are in the ground model.
\end{proof}

Following the notation above, let $V$ be a model of ZF, $W\in V$ an infinite set, and consider the model $V(A, T)$ constructed by the two steps above.
Consider the poset $\Q_1=\Q_1(A)\in V(A,T)$ be the poset of all finite partial functions from $A$ to $\{0,1\}$.
The $\Q_1$ generic object is identified with a subset of $A$. The poset $\Q$ above is isomorphic to the finite support product of $\omega$-many copies of $\Q_1$.
For a finite $\bar{a}\subset A$ and a condition $p\in\Q_1$, let $p\restriction \bar{a}$ be the restriction of $p$ to the domain $\bar{a}$.

Fix a natural number $d$ and let $\Q_d=\Q_1^d$, adding $d$ mutually generic subsets of $A$.
For $p=\seqq{p(i)}{i<d}$ and $\bar{a}\subset A$, let $p\restriction\bar{a}$ be the condition $\seqq{p(i)\restriction\bar{a}}{i<d}$.
\begin{lemma}\label{lemma: basic Monro lemma}
Let $\bar{a}\subset A$ be finite, and $\bar{x}$ a sequence such that $x_i\in a_i$.
Let $\psi$ be a formula, $v\in V$, and $p\in\Q_d$ such that
\begin{equation*}
    V(A,T)\models p\force \psi(A,T,\bar{x},v).
\end{equation*}
Then
\begin{equation*}
        V(A,T)\models p\restriction\bar{a}\force \psi(A,T,\bar{x},v).
\end{equation*}
\end{lemma}
\begin{proof}
We may assume that for each $i<d$ the domain of $p(i)$ is of the form $\bar{a}\cup\bar{b}$ where $\bar{b}\cap\bar{a}=\emptyset$.
Write $\bar{b}=\seq{b_1,\dots,b_n}$, distinct elements.
Given a sequence $\bar{c}=\seq{c_1,\dots,c_n}\subset A$ of distinct elements, disjoint from $\bar{a}$, define $p[\bar{c}]$ to be ``$p$ with $\bar{c}$ replaced for $\bar{b}$''. That is,
\begin{itemize}
    \item the domain of each $p(i)$ is $\bar{a}\cup\bar{c}$;
    \item for $i<d$ and $j=1,\dots,n$, $p[\bar{c}](i)(c_j)=1\iff p(i)(b_j)=1$;
    \item $p[\bar{c}]\restriction \bar{a}=p\restriction\bar{a}$.
\end{itemize}
Note that the statement ``$p[\bar{c}]\force\psi(A,T,\bar{x},v)$'' can be written as $\chi(A,T,\bar{x},\bar{c},v)$ for some formula $\chi$.
Applying Lemma~\ref{lemma: stronger basic tree step} we conclude that for any sequence $\bar{c}=\seq{c_1,\dots,c_n}$ of distinct members of $A$, disjoint from $\bar{a}$, if $\pi(c_j)=\pi(b_j)$ for $j=1,\dots,n$, then $p[\bar{c}]\force\psi(A,T,\bar{x},v)$.

Note that for any condition $q$ below $p\restriction\bar{a}$, if $\bar{c}$ is as above and disjoint from the domain of $q$, then $p[\bar{c}]$ is compatible with $q$.
The conclusion now follows as the conditions forcing $\psi(A, T,\bar{x},v)$ are pre-dense below $p\restriction\bar{a}$.
\end{proof}
Let $\seqq{x_n}{n<\omega}$ be $\Q(A)$ generic over $V(A)$.
Let $B=\set{x_n}{n\in\omega}$.
\begin{lemma}
Let $d,n$ be natural numbers and $r\colon d\times n\to\{0,1\}$ a function.
Then for any formula $\phi$ there is a formula $\psi$ such that for any finite sequence $\bar{y}=\seq{y_0,\dots,y_{d-1}}$ of distinct elements from $B$, any sequence $\bar{a}=\seq{a_0,\dots,a_{n-1}}$ of distinct members of $A$ and $\bar{x}=\seq{x_0,\dots,x_{n-1}}$ such that $x_i\in a_i$, if for all $i<n$ and $j<d$
\begin{equation*}
    a_i\in y_j\iff r(i,j)=1,
\end{equation*}
then
\begin{equation*}
    V(A,T)(B)\models\phi(A,T,B,\bar{x},\bar{y})\iff V(A,T)\models\psi(A,T,\bar{x})
\end{equation*}
\end{lemma}
\begin{proof}
By Lemma~\ref{lemma: basic permutation} there is a formula $\zeta$ such that 
\begin{equation*}
    V(A,T)(B)\models\phi(A,T,B,\bar{x},\bar{y})\iff V(A,T)[\bar{y}]\models\zeta(A,T,\bar{x},\bar{y}).
\end{equation*}
Note that $\bar{y}$ is $\Q_d$-generic over $V(A,T)$.
Let $p\in\Q_d$ be $\bar{y}\restriction\bar{a}$. That is, the domain of $p(i)$ is $\bar{a}$ for $i<d$, and $p(i)(a_j)=1\iff y_i(a_j)=1$.
By Lemma~\ref{lemma: basic Monro lemma},
\begin{equation*}
    V(A,T)[\bar{y}]\models\zeta(A,T,\bar{x},\bar{y})\iff V(A,T)\models p\force \zeta(A,T,\bar{x},\dot{\bar{y}}).
\end{equation*}
The lemma now follows with $\psi(A,T,\bar{x})$ defined as ``$q\force_{\Q_d} \zeta(A,T,\bar{x},\dot{\bar{y}})$ for the condition $q$ defined as $q(i)(a_j)=1\iff r(i,j)=1$''. (Note that $a_j$ is definable from $x_j$.)
\end{proof}

Let $\hat{a}_n=[x_n]=\set{y\subset A}{y\Delta x_n\textrm{ is finite}}$, and $\hat{A}=\set{\hat{a}_n}{n\in\omega}$.
$\hat{A}$ is to $V(A,T)$ as $A$ is to $V$, according to Lemma~\ref{lemma: basic step} above.
\begin{cor}
The lemma above is also true if $V(A,T)(B)$ is replaced with $V(A,T)(\hat{A})$ and $\phi(A,T,B,\bar{x},\bar{y})$ replaced with $\phi(A,T,\hat{A},\bar{x},\bar{y})$.
\end{cor}
\begin{proof}
Since $\hat{A}$ is definable from $B$, a statement of the form $\phi^{V(A,T)(\hat{A})}(A,T,\hat{A},\bar{x},\bar{y})$ can be written as $\chi^{V(A,T)(B)}(A,T,B,\bar{x},\bar{y})$ for some formula $\chi$. By the lemma above there is a formula $\psi$ corresponding to $\chi$, which is as required.
\end{proof}
Note that $A$ is definable from $\hat{A}$ as $A=\bigcup\hat{A}$.
Working in $V(A,T)(\hat{A})$, let $\hat{\T}$ be the poset for adding a function $\hat{\pi}\colon\hat{A}\to A$ by finite conditions. Fix a generic $\hat{\pi}$ and let $\hat{T}$ be the tree structure between $A$ and $\hat{A}$ defined from $\hat{\pi}$.
Then the model $V(A,T)(\hat{A},\hat{T})$ is to $V(A,T)$ as $V(A,T)$ was to $V$, according to Lemma~\ref{lemma: basic tree step}.

\begin{lemma}\label{lemma: basic step full indiscernibility}
Let $d,n$ be natural numbers and $r\colon d\times n\to\{0,1\}$ a function. Then for any formula $\phi$ there is a formula $\psi$ such that for any finite sequence $\hat{\bar{a}}=\seq{\hat{a}_0,\dots,\hat{a}_{d-1}}$ of distinct members of $\hat{A}$, for any $\hat{\bar{y}}=\seq{\hat{y}_0,\dots,\hat{y}_{d-1}}$ with $\hat{y}_i\in \hat{a}_i$, any sequences $\bar{a}=\seq{a_0,\dots,a_{n-1}}$ and $\bar{b}=\seq{b_0,\dots,b_{d-1}}$ from $A$, and $\bar{x}=\seq{x_0,\dots,x_{n-1}}$ such that $x_i\in a_i$, if
\begin{itemize}
    \item for all $i<n$ and $j<d$, $a_i\in y_j\iff r(i,j)=1$ and
    \item $\hat{\pi}(\hat{a}_j)=b_j$ for $j<d$, then
\end{itemize}
\begin{equation*}
    V(A,T)(\hat{A},\hat{T})\models\phi(A,T,\hat{A},\hat{T},\bar{x},\bar{y})\iff V(A,T)\models\psi(A,T,\bar{x},\bar{b})
\end{equation*}
\end{lemma}
The proof is analogous to the proof of Lemma~\ref{lemma: basic tree step}.

\begin{proof}[Proof of Theorem~\ref{thm: full indiscernibility}]
The proof proceeds by induction on the ordinals. The limit cases follow as in subsection \ref{subsection: tree step} and \ref{subsection: limit-step}. The successor step follows from Lemma~\ref{lemma: basic step full indiscernibility}, as in Section~\ref{section: basic lemmas for successor step}.
\end{proof}

\begin{prop}\label{prop: A-theta-D-finite-in-extensions}
Let $\theta$ be a limit ordinal, $\P$ a poset in $V(A,T)$, $\P\subset V(A,T)_\theta$. Then $A_\theta$ is Dedekind-finite in the extension $V(A, T)[\P]$.
\end{prop}
In particular, $A_\theta$ is Dedekind-finite in any $\mathrm{Col}(\omega,A_{<\theta})$-generic extension of $V(A,T)$.
\begin{proof}
Let $p\in\P$ be a condition and $\tau\in V(A,T)$ a $\P$-name. Fix $\bar{a},\bar{x}$ such that $\tau$ is definable from $A, T,\bar{a}$ and a parameter in $V$, where $\bar{a}$ is a finite sequence of elements from $A$ and $\bar{x}$ is a finite sequence of representatives.
Suppose $b,c\in A_\theta$ have the same type over $p,\bar{a},\bar{x}$. It follows from the Indiscernibility Theorem~\ref{thm: full indiscernibility} that for any formula $\phi$,
\begin{equation*}
    p\force \phi(A,T,\tau,b)\iff p\force \phi(A,T,\tau,c).
\end{equation*}
Suppose for contradiction that $\tau$ is a $\P$-name for an $\omega$-sequence of distinct elements in $A_\theta$. Then there must be some $b\in A_\theta$, not in $\bar{a}$, and a condition $p\in\P$ forcing that $\tau(n)=b$, for some $n\in\omega$.
Find now $c\in A_\theta$, $b\neq c$, such that $b$ and $c$ have the same type over $p,\bar{a},\bar{x}$. This is possible since $p$ is in $V(A,T)_{\theta}$.
Now $p\force \tau(n)=c$ as well, a contradiction.
\end{proof}

\begin{cor}\label{cor: failure of DC in set forcing}
\begin{enumerate}
    \item $\mathrm{DC}[\R]$ fails in the extension $V(A,T)[\mathrm{Col}(\omega,A_{<\theta})]$, for any limit ordinal $\theta$. 
    \item DC fails in any set-forcing extension of $V(A,T)$.
\end{enumerate}
\end{cor}
\begin{proof}
Let $\theta$ be a limit ordinal. In the $\mathrm{Col}(\omega,A_{<\theta})$ extension, $A_{<\theta}$ is countable, and therefore $A_\theta$ can be identified as a set of reals. 
By Proposition~\ref{prop: A-theta-D-finite-in-extensions}, $A_\theta$ is in fact a Dedekind-finite set of reals in the extension, so $\mathrm{DC}[\R]$ fails.

For any poset $\P$, it follows from Proposition~\ref{prop: A-theta-D-finite-in-extensions} that for large enough $\theta$, $A_\theta$ will be Dedekind-finite in any $\P$-generic extension (e.g.\ take $\theta$ such that $\mathbb{P}\subseteq V(A,T)_\theta$). In particular, $\DC$ fails in any such extension.
\end{proof}

\section{Intermediate models are symmetric extensions}
By a Theorem of Usuba, \cite{Usuba2021geology}, $V(A_\alpha,T_\alpha)$ is a symmetric extension of $V$, possibly using a forcing which is not $\mathbb{P}_\alpha$. Using the permutations from $\mathcal{G}_\alpha$ (see Definition~\ref{definition:group-of-automorphisms-for-P_alpha}) we can present the $\alpha$-th stage of the construction, $V(A_\alpha, T_\alpha)$, as a symmetric extension using the forcing $\P_\alpha$, by taking the filter of groups to consist of the stabilizers of the members of the transitive closure of $A_\alpha$, $\mathcal{F}_\alpha$, or a minor variant of. We will need that for the last step of the construction in $L[c]$, in which it will be important that each step in the construction, even though constructed in a way very far from being a symmetric extension, is actually a symmetric extension using some generic filter that exists in a further extension.

Let $\mathbb{T}$ be the following forcing defined in $V(A_\alpha, T_\alpha)$. A condition $q\in \mathbb{T}$ is a finite subset of $\{((f^\beta_t)_*\dot{x}_{\beta,n}, y) \mid \exists a \in A_\beta, y \in a,\, n < \omega,\, t\in [\omega]^{<\omega}, \beta \leq \alpha \text{ successor}\} \cup \{\dot x_{\beta,n} \mid n < \omega,\, \beta \leq \alpha \text{ limit }\}$\footnote{We will use automorphisms of the form $f^\beta_{\bar{t}}$ where $\bar{t}$ is a sequence of $n$ empty sets and possibles arbitrary finite set at step $n$. To simplify notations, we temporarily write the value $\bar{t}_n$ instead of the full sequence $\bar{t}$ when this automorphism is applied to $\dot{x}_{\beta,n}$, as the other values do not change the name.} such that
\begin{itemize}
\item $q$ an injective function. 
\item If $\beta$ is successor ordinal and $f_t^\beta \dot x_{\beta,n}, f_s^\beta \dot x_{\beta,n} \in \dom q$ then 
\begin{itemize}
\item $q((f_t^\beta)_* \dot x_{\beta, n}) \triangle q((f_s^\beta)_* \dot x_{\beta, n}) = c$ is finite.
\item for every $m \in t \triangle s$, there is $u_m$ such that $(f^{\beta - 1}_{u_m})_* \dot x_{\beta - 1, m} \in \dom q$, and $\{[\dot x_{\beta - 1 , m})] \mid m \in t \triangle s\} = c$.
\end{itemize}
\item If $y, z \in \range q$ and $y \triangle z$ is finite, then $q^{-1}(y) = (f^\beta_t)_* \dot x_{\beta,n}$ and $q^{-1}(z) = (f^\beta_s)_* \dot x_{\beta, n}$ for some $t,s \in [\omega]^{<\omega}$, $n < \omega$ and $\beta \leq \alpha$.
\end{itemize}
\begin{lemma}\label{lem;T-is-homogeneous}
$\mathbb{T}$ is weakly homogeneous.
\end{lemma}
\begin{proof}
Let $q_0, q_1\in \mathbb{T}$. Note that every automorphism $\pi$ of $\mathbb{P}_\alpha$ defines an automorphism $L_\pi$ of $\mathbb{T}$ by $L_\pi(\{(\dot\tau_i, a_i) \mid i < n\}) = \{\pi_*\dot\tau_i, a_i) \mid i < n\}$. During the proof of this lemma, we will not distinguish between $\pi$ and $L_\pi$, but this distinction will play a role later.

We want to extend $q_0$ and $q_1$ and find an automorphism moving the extensions to compatible conditions. For every ordinal $\beta$, let $q \restriction \beta$ be the restriction of $q$ to names for the form $(f^\gamma_t)_* \dot{x}_{\gamma, n}$ for $\gamma < \beta$.

Let us assume that $q_0 \restriction \beta$ is compatible with $q_1 \restriction \beta$. We will find an extension of the condition and an automorphism sending the extensions to members of the forcing $\mathbb{T}$ such that their restriction to $\beta + 1$ is compatible. 

If $q_0 \restriction \beta + 1 = q_0 \restriction \beta$ and 
$q_1 \restriction \beta + 1 = q_1 \restriction \beta$ we do not modify the conditions. 

Let us assume that this is not the case. There are two cases that we need to consider.

Let us assume that either $\beta$ is a limit ordinal or there is no $y \in \bigcup A_\beta \cap \range q_0$ and $z\in \range q_1$ such that $y \triangle z$ is finite and non-empty. In this case, by applying a permutation on the indexing of $\beta$, $g_\sigma$ we obtain compatibility. 

Otherwise, for each such $y,z$ let $n_y$ be the natural number such that $q_0^{-1}(y) = (f^\beta_t)_* \dot{x}_{\beta,n}$ for some $t$ and $m_z$ be the corresponding number for $q_1^{-1}(z)$. Note that the map sending $n_y$ to $m_z$ depends only on the equivalence class of $y$ and $z$ and thus it is a well-defined injection. Let $\sigma$ be a finite permutation extending $n_y \mapsto m_z$ for each such pair $y,z$. By applying $g_\sigma^\beta$ on $q_0$, we may assume that $n_y = m_z$ for each such $y,z$. 

Let us look at the collection of all translations $t$ such that $(f^\beta_t)_* \dot{x}_{\beta,n} \in \dom q_0$ and similarly all translations $s$ such that $(f^\beta_s)_* \dot{x}_{\beta.n} \in \dom q_1$. Without loss of generality (by applying an automorphism of the form $f^\beta_t$ and extending the conditions) we may assume that $\dot \dot{x}_{\beta,n} \in \dom q_0 \cap \dom q_1$ and $q_0(\dot x_{\beta,n}) \triangle  q_1(\dot x_{\beta,n})$ is finite. 
Moreover, we may assume that for any translation $t$ such that $(f^\beta_t)_* \dot x_{\beta,n} \in \dom q_0$, and $k\in t$, also $(f^\beta_{\{k\}})_* \dot x_{\beta,n} \in \dom q_0$ and thus some translation of $\dot{x}_{\beta - 1,k}$ is in $\dom q_0$, and the same for $q_1$. 
We will assume also that if some translation of $\dot{x}_{\beta - 1,k}$ is in $\dom q_0$ then $(f_{\{k\}}^\beta)_* \dot{x}_{\beta, n} \in \dom q_0$ and the same holds for $q_1$. 

By the induction hypothesis, we may assume that $q_0 \restriction \beta$ is compatible with $q_1 \restriction \beta$.

Now, let us look at $c=q_0(\dot x_{\beta, n}) \triangle  q_1(\dot x_{\beta,n})$. Let $c_0$ be the subset of $c$ consisting of elements that are already equivalence classes of members of the range of $q_0$ and let $c_1 = c \setminus c_0$. Let $t_0$ be the corresponding translation by indexes of elements of $c_0$ (which are already determined) and let $t_1$ be a translation disjoint to all previous translations, $|t_1| = |c_1|$. Finally, apply $f^\beta_{t_0 \cup t_1}$ on $q_0$.   

Continue this way, after finitely many non-trivial steps, we obtain compatible conditions.
\end{proof}

Let $G$ be a generic filter generating $\dot{A}_\alpha, \dot{T}_\alpha$. Let $H$ be a filter consisting of finite subsets of $\{(f_t\dot{x}_n^\beta, (f_t\dot{x}_n^\beta)^G) \mid n < \omega, \beta \leq \alpha, t\in [\omega]^{<\omega}\}$. 
\begin{lemma}
$H$ is $V(A_\alpha,T_\alpha)$ generic for $\mathbb{Q}$. 
\end{lemma}
\begin{proof}
Let $D \in V(A_\alpha,T_\alpha)$ dense open.  

Let $p\in \mathbb{P}_\alpha$. We need to show that there is $p' \leq p$ that forces $H \cap D \neq \emptyset$. Since $D \in V(A_\alpha, T_\alpha)$, it is definable using finitely many parameters, $\bar{a}$. 

Fix a name for a condition $\bar{r}\in \dot{H}$ with a range covering $\bar{a}$ and all the elements of $\trcl(A_\alpha)$ mentioned by $p$. Let $r$ be a name for a stronger condition in $D$, as forced some by $p'$. 

Let $\pi$ be an automorphism of $\mathbb{P}_\alpha$ in $\mathcal G_\alpha$. By extending $p'$, if necessary, we may assume that $r = \{ \langle \check\tau_i, \dot{b}_i\rangle \mid i < n\}$. So, $\pi_*(r) = \{ \langle \check\tau_i, \pi_*\dot{b}_i\rangle \mid i < n\}$, which is different than $L_\pi(r)$. Let us denote this condition by $R_\pi(r)$. Nevertheless, the proof of Lemma \ref{lem;T-is-homogeneous} works without significant modifications and show that the automorphism $R_\pi$ witnesses the homogeneity of the forcing as well. Unlike $L_\pi$, the correspondence between $\pi$ and $R_\pi$ does not exist in $V(A_\alpha,T_\alpha)$. 

Let $\tilde{r}$ be a condition in $H$ with $\dom \tilde r = \dom r$. The proof of Lemma \ref{lem;T-is-homogeneous} illustrates that the obtained automorphism $\pi$ fixes $p$ and $\bar{a}$ and that when extending the $\tilde{r}$ side we may keep the condition in $H$.\footnote{Whenever we extend a condition during the proof of Lemma \ref{lem;T-is-homogeneous}, we only require it to contain additional elements in its domain, which can be obtained using a condition in $H$.} We obtain that after the extensions the automorphism outright sends one extension to the other. 

By applying the obtained automorphism $\pi$ and strengthening $\tilde{r}, r$ we obtain we have $\pi(p')$ forces $\pi_*(r) = \tilde{r}$ to be in $D$, and thus $p' \Vdash H \cap D \neq \emptyset$.  
\end{proof}

In \cite{Grigorieff-1975}, Grigorieff proved that for every symmetric extension $M \subseteq V[G]$ there is a homogeneous forcing notion $\mathbb T \in M$ and an $M$-generic filter $T \subseteq \mathbb T$ such that $M[T] = V[G]$. It is not difficult to see that the existence of such a homogeneous forcing implies that $M$ is a symmetric extension, using another theorem of Grigorieff. 

\begin{thm}\label{thm: M_alpha is symmetric extension}
    $V(A_\alpha, T_\alpha)$ is a symmetric extension by $\langle \mathbb{P}_\alpha, \tilde{\mathcal{G}}_\alpha, \tilde{\mathcal{F}}_\alpha\rangle$, where $\tilde{G}_\alpha \supseteq \mathcal{G}_\alpha$ and for every subgroup $X\in \tilde{\mathcal{F}}_\alpha$, $X \cap \mathcal G \in \mathcal{F}_\alpha$.
\end{thm}
\begin{proof}
Let us look at $\mathbb{T}$ and let $H$ be a generic filter for $\mathbb{T}$. As it is a weakly homogeneous forcing notion, 

\[V(A_\alpha,T_\alpha) = \hod^{V(A_\alpha, T_\alpha)}_{V, \trcl(\{A_\alpha, T_\alpha\})} \supseteq \hod^{V(A_\alpha, T_\alpha)[H]}_{V, \trcl(\{A_\alpha, T_\alpha\})}\]
Let $H$ be the generic filter derived from $G$. Clearly, $V(A_\alpha,T_\alpha)[H] = V[G]$, so we obtain 
\[\hod^{V(A_\alpha, T_\alpha)[H]}_{V, \trcl(\{A_\alpha, T_\alpha\})} = \hod{V[G]}_{V, \trcl(\{A_\alpha, T_\alpha\})}\]
By a theorem of Grigorieff, this last model is a symmetric extension using the forcing notion $\mathbb{P}_\alpha$, with group of automorphism $\tilde{\mathcal{G}} \leq \mathrm{Aut}(\mathcal{B}(\P_\alpha))$ and a filter of groups stabilizing each member of the transitive closure of $A_\alpha$.
\end{proof}

\section{Shuffling}\label{section:shuffling}
The following section is deeply related to the work of Grigorieff and Karagila on symmetric extensions and iterations, \cite{Grigorieff-1975, Karagila-iterating-2019}. Indeed, the relationship between symmetric systems and their corresponding extensions, symmetrically generic filters, and how to transform them into proper generic filters is the main motivation for the results in this section. 

The following shuffling lemma will be used repeatedly in Section~\ref{section: construction in Cohen extension}, when ``sufficiently generic'' filters will be constructed inside a Cohen-real extension of $V$. The basic idea is that, given some iteration, it will be relatively simple to construct a filter which is generic for each bounded part of the iteration. To get a generic for the full iteration, we will shuffle it using a generic sequence of automorphisms. Furthermore, we will restrict ourselves to a subgroup of automorphisms, those preserving the key objects of our construction, as constructed in Section \ref{subsubsection: permutations}.

A simple example is as follows. Let $\mathbb{C}$ be Cohen forcing for adding a single subset of $\omega$, and $\P$ the finite support product of countably many copies of $\mathbb{C}$. Suppose $z=\seqq{z_n}{n<\omega}$ is a sequence so that for each $n$, $\seq{z_0,\dots,z_{n-1}}$ is generic for $\mathbb{C}^n$. It is not necessarily true that $z$ is $\P$-generic. 
However, by making finite changes to each coordinate, we can get a sequence $\seqq{z'_n}{n<\omega}$ which is $\P$-generic.
Specifically, force by finite approximations a sequence $\seqq{g_n}{n<\omega}$, where each $g_n$ is an automorphism of $\mathbb{C}$ flipping finitely many coordinates, and let $z'_n=g_n\cdot z_n$. Then $\seqq{z'_n}{n<\omega}$ is $\P$-generic.
Note that in this case, the sequence of mod-finite equivalence classes is the same, $\seqq{[z_n]}{n<\omega}=\seqq{[z'_n]}{n<\omega}$.

More generally, we may force a generic automorphism of a poset as follows.
\begin{defn}\label{defn: shuffling poset}
Let $\P$ be a poset, $S\leq\mathrm{Aut}(\P)$ a subgroup of automorphisms of $\P$, and $\mathcal{F}$ a filter over $S$. Assume that for each $p\in\P$, the stablizer of $p$ in $S$ is large, $S_p=\set{g\in S}{g(p)=p}\in\mathcal{F}$.\footnote{In this case, the condition is \emph{tenacious}. This assumption is harmless, see \cite[Appendix A]{Karagila-iterating-2019}}
Define
\begin{equation*}
\mathbb{S}(\P,S,\mathcal{F})=\set{(g,X)}{g\in S,\,X\in\mathcal{F}}.    
\end{equation*}
For two pairs $(h,Y)$ and $(g,X)$ in $\mathbb{S}(\P,S,\mathcal{F})$, say that $(h,Y)$ extends $(g,X)$ if 
\begin{itemize}
    \item $Y\subset X$, and
    \item $h=r\cdot g$ for some $r\in X$.
\end{itemize}
A generic filter for $\mathbb{S}(\P, S,\mathcal{F})$ naturally defines a permutation of $\P$ as follows. Suppose $H\subset \mathbb{S}(\P,S,\mathcal{F})$ is a generic filter. Define $\pi\colon \P\to\P$ by
\begin{equation*}
    \pi(p)=q\iff (\exists g\in S)\, (g,S_q)\in H\wedge g(p)=q.
\end{equation*}
When $\P,S,\mathcal{F}$ are clear from context, we will denote by $\dot{\pi}$ the canonical $\mathbb{S}(\P,S,\mathcal{F})$-name for the permutation $\pi$. When a generic filter $H\subset\mathbb{S}(\P,S,\mathcal{F})$ is fixed we will refer to $\dot{\pi}[H]$ as $\pi$.
\end{defn}

\begin{claim}
Following the notation of Definition~\ref{defn: shuffling poset}, for any generic filter $H\subset \mathbb{S}(\P,S,\mathcal{F})$ over $V$, $\pi$ is a well defined automorphism of $\P$, in $V[H]$.
\end{claim}
For $\mathcal{S} = \langle \mathbb{P}, \mathcal G, \mathcal F\rangle$  symmetric system, so $\mathcal{G} \leq \mathrm{Aut}(\mathbb P)$ and $\mathcal F$ is a subgroups filter, we denote by $\mathbb{S}(\mathcal S)$ the forcing $\mathbb{S}(\mathbb{P}, \mathcal{G}, \mathcal{F})$.

\begin{ex}\label{example: generic automorphism for product}
Suppose $\P$ is the finite support product of $\Q_n$, $n<\omega$, and $S_n\leq \mathrm{Aut}(\Q_n)$ is a subgroup of automorphisms of $\Q_n$. 
The product group $\prod_n S_n$ naturally acts on $\P$, coordinate-wise, so we identify it as a subgroup of $\mathrm{Aut}(\P)$. 
Let $S$ be the subgroup of $\prod_n S_n$ of all finite support sequences.

In this case, a natural way to add a generic automorphism is by finite approximations. Let $\mathbb{S}$ be the poset $S$, ordered by end-extensions. A generic for $\mathbb{S}$ is identified with a generic sequence $\vec{s}\in\prod_n S_n$.

Let $\mathcal{F}$ be the filter over $S$ generated by sets of the form \[\set{g\in S}{g\textrm{ fixes the first }n\textrm{ coordinates}},\] for some $n$.
Then $\mathbb{S}$ and $\mathbb{S}(\P,S,\mathcal{F})$ are forcing equivalent. Moreover, this equivalence identifies the generic permutation $\pi$ with the generic permutation $\vec{s}$.
\end{ex}

\begin{defn}\label{definition:symmetrically-generic-filter}
Fix $\P,S,\mathcal{F}$ as in Definition~\ref{defn: shuffling poset}. Let $G \subseteq \P$ be a filter. We say that $G$ is \emph{symmetrically generic} relative to $\langle \mathbb{P}, S, \mathcal{F}\rangle$ if for any dense open $D\subset\P$ in $V$ and any $X\in\mathcal{F}$, $(X\cdot D)\cap G\neq \emptyset$, where $X\cdot D=\set{g(d)}{g\in X,\,d\in D}$.
\end{defn}
\begin{lemma}[The Shuffling Lemma]\label{lem;the shuffling lemma}
Let $V \models \ZF$. Fix $\P, S,\mathcal{F}$ as in Definition~\ref{defn: shuffling poset} and let us assume that $G\subset \P$ is a symmetrically generic filter. 

Let $H\subset \mathbb{S}(\P,S,\mathcal{F})$ be a generic filter over $V$ and let $\pi$ be the corresponding generic automorphism of $\P$ as in Definition~\ref{defn: shuffling poset}.
Then $\pi^{-1}G=\set{\pi^{-1}(p)}{p\in G}$ is a $\P$-generic filter over $V$.
\end{lemma}
\begin{proof}
Fix a dense open set $D\subset\P$ in $V$. Let $(g,X)$ be a condition in $\mathbb{S}(\P,S,\mathcal{F})$. We need to find an extension of $(g,X)$ forcing that $\dot{\pi}^{-1} G\cap D\neq\emptyset$.
Since $D\subset\P$ is open dense, and $g$ is an automorphism of $\P$, then $g\cdot D=\set{g(d)}{d\in D}\subset\P$ is open dense as well. By assumption, there is some $r\in X$ and $d'\in g\cdot D$ so that $r(d')\in G$. Fix $d\in D$ for which $d'=g(d)$, so $r\cdot g(d)\in G$.
Let $Y=X\cap S_{r\cdot g(d)}$. Then $(r\cdot g,Y)$ is an extension of $(g,X)$ forcing that $\dot{\pi}(d)\in G$, and so that $\dot{\pi}^{-1}G\cap D\neq\emptyset$.
\end{proof}
The following claim is a combination of Example \ref{example: generic automorphism for product} with Karagila's analysis of iteration of automorphisms from \cite[Section 3]{Karagila-iterating-2019}. The main obstacle that we must overcome when moving from products to iterations is that even a two-step iteration of weakly homogeneous forcing notions does not have to be weakly homogeneous. The standard solution to this problem is to assume that enough of the structure of the iteration as well as the homogeneity is preserved by the automorphisms, \cite{Dobrinen-Friedman-2008}. 

Unlike the Shuffling Lemma, which works in $\ZF$, the following claim requires the axiom of choice, which is used in the definition of the application of an automorphism that exists in a generic extension to a name.\footnote{While the claim in its full generality requires the axiom of choice, we will apply it to cases in which the forcing notion, as well as the subgroup of automorphisms, are well ordered.}  

\begin{defn}\label{definition:suitable-for-omega-iteration}
Let $\P = \P_\omega$ be a finite support iteration of $\langle \P_n, \dot{\Q}_n \mid n < \omega\rangle$. Let $\dot{S}_n$ be a $\P_n$-name, forced by the trivial condition to be a subgroup of the automorphism group of $\dot{\Q}_n$, witnessing its weak homogeneity, namely it is forced by the trivial condition that for every $\dot{q}_0, \dot{q}_1 \in \dot{\Q}_n$ there is $\dot\sigma \in \dot{S}_n$ such that $\dot{\sigma}(\dot{q}_0)$ is compatible with $\dot{q}_1$. 

We say that $\P, \langle \dot{S}_n \mid n < \omega\rangle$ is \emph{suitable for iteration} if the following condition holds.

Let us work, temporarily, in the generic extension by $\mathbb{P}_n$. By induction on $m > n$, any automorphism $\sigma\in S_n$ induces an automorphism on the iteration $\mathbb{Q}_n \ast \mathbb{Q}_{n+1}\ast \cdots \ast \mathbb{Q}_m$ that we will denote (temporarily) by $\sigma_{n,m}$. Indeed, $\sigma_{n,n} = \sigma$ and $\sigma_{n,m+1}(\langle q_n,\dot{q}_{n+1},\dots, q_{m+1}\rangle = \sigma_{n,m}(\langle q_n,\dot{q}_{n+1},\dots, \dot{q}_{m}\rangle) ^\smallfrown \langle (\sigma_{n,m})_*(\dot{q}_{m+1})\rangle$. 
\footnote{Recall that given an automorhpism $\sigma$ of a forcing notion $\mathbb R$ it can be extended to an operation $\sigma_\ast$ on $\mathbb{R}$-names, where $\sigma_\ast$ is defined recursively on the names by $\sigma_* (\dot{x}) = \{(\sigma(r), \sigma_*(\dot{y})) \mid (r,\dot y)\in \dot x\}$ for an $\mathbb{R}$-name $\dot{x}$.}

Let us assume further that for every automorphism $\sigma$ in $\dot{S}_n$ it is forced that for all $m > n$, $\sigma_{n,{m-1}}(\dot{Q}_m) = \dot{Q}_m,\, \sigma_{n,{m-1}}(\leq_{\dot{Q}_m})=\leq_{\dot{Q}_m}$ and $\sigma_{n,m-1}(\dot{S}_m)= \dot{S}_m$. 
\end{defn}
\begin{claim}\label{claim:shuffling-omega-iteration}
Work in $\ZFC$\footnote{When working in $\ZF$, we will apply Claim~\ref{claim:shuffling-omega-iteration} by moving to a generic extension in which a sufficiently large initial segment of the universe is well ordered. See Section~\ref{subsection: V[c] limit stage}}. Let $\P = \P_{\omega}$ be a finite support iteration and let $\langle \dot{S}_n \mid n < \omega\rangle$ be a sequence of witnesses for the homogeneity of each step in the iteration, which is suitable for iteration as in Definition \ref{definition:suitable-for-omega-iteration}.

Let $G \subseteq \P$ be a filter such that $G \restriction n \subseteq \P_n$ is $V$-generic for all $n$. Then, in a forcing extension of $V[G]$ there is a generic automorphism $\pi$ such that $\pi(G)$ is $V$-generic.   
\end{claim}
The automorphism $\pi$ belongs to the inverse limit (in some sense) of the $S_n$. One should keep in mind that the structure of the group which is composed of the $S_n$ is not a direct product, but rather a semi-direct one, but we will not pursue this computation here. 
\begin{proof}
For $q\in\mathbb{Q}$, recall that $\supp q$ is the set of coordinates $n$ such that $q(n)$ is not forced by $1_{\mathbb{P}_n}$ to be $1_{\dot{\mathbb{Q}}_n}$.
Let $\dot{\P}^n$ be the $\P_n$ name for the finite support iteration of  $\langle \dot{\Q}_m \mid n \leq m < \omega\rangle$. There is a canonical isomorphism of forcing notions $\iota\colon \P \cong \P_n\ast \dot{\P}^n$. Note that there is a dense subset of $\P_n \ast \dot{\P}^n$ of conditions $(p, \dot q)$ for which there is a finite set $a$ such that $p \Vdash \supp (\dot{q}) = \check{a}$.

Every automorphism in $\dot{S}_n$ induces an automorphism of $\dot{\P}^n$, by applying 
\[\sigma_{n,\omega}(\langle \dot{q}_m \mid m \geq n\rangle) = \langle \sigma_{n,m}(\dot{q}_m) \mid m \geq n\rangle.\]
By a slight abuse of notation, we identify $\dot{\sigma}$ with $\sigma_{n,\omega}$. Under this identification, every element of $\dot{S}_n$ is in fact a $\P_n$-name for a  \emph{support preserving} automorphism of $\dot{\P}^n$ (namely, $\Vdash_{\P_n} \supp \dot q = \supp \dot\sigma(\dot q)$ for every $\dot\sigma \in \dot{S}_n$ and $\dot{q} \in \dot{\P}^n$). 

\begin{lemma}
There is a subgroup $T_n \leq \mathrm{Aut}(\P)$, such that 
$\Vdash_{\mathbb{P}_n}$ ``there is a surjective homomorphism $T_n \rightarrow \dot{S}_n$''.  
\end{lemma}
\begin{proof}
We work in the dense subset of $(p, \dot q)\in \P_n \ast \P^n$ for which there is a finite set $a$ such that $p \Vdash \supp (\dot{q}) = \check{a}$. For every name $\dot{\sigma}$ for an element of $\dot{S}_n$ (considered as an automorphism of $\dot{\mathbb{P}}^n$), we define $\dot{\sigma}(\dot{q})$ to be the condition forced to be the image of $\sigma$ on $\dot{q}$. Let $t(\dot{\sigma}) = \iota^{-1}(id_{\mathbb{P}_n} \ast \dot{\sigma}) \iota$ and let $T_n$ be the image of $t$. 

Note that
\begin{equation*}
\Vdash_{\mathbb{P}_n} \dot{\sigma} = \dot{\sigma}'   \textrm{ if and only if } t(\dot{\sigma}) = t(\dot{\sigma}'),
\end{equation*}
so $T_n$ is a set. Similarly, $T_n$ is a subgroup of the automorphisms of $\P$. Finally, given a generic filter $H \subseteq \P_n$, we can define a map $t(\dot{\sigma}) \mapsto \dot{\sigma}^{H}$, which is a group homomorphism and is surjective by the definition.  
\end{proof}
Let us remark that for the definition of $t$ to make sense (without moving to the Boolean completion), we must use the assumption that the automorphisms are support-preserving. Namely, the support of the condition $\dot\sigma(\dot{q})$ equals the support of $\dot{q}$ and therefore a concrete finite set.

One can apply the lemma in the generic extension by $\P_n$, and obtain a group $(T_m)^{V[\P_n]}$, for $m\geq n$, which projects to $\dot{S}_m^{V[\P_m]}$ in any further $\dot{\P}^n$ generic extension.  
Working in $V[\P_n]$, let $S^n$ be the group generated by $\langle (T_m)^{V[\P_n]} \mid n \leq m < \omega\rangle$. 
\begin{lemma}
Working in $V[\P_n]$, $S^n$ witnesses the weak homogeneity of $\P^n$, that is, for any two conditions $p,q\in \P^n$ there is some $s\in S^n$ such that $s(p)$ is compatible with $q$.
\end{lemma}
\begin{proof}
Let $p, q$ be conditions in $(\dot{\P}^n)^{V[\P_n]}$, and let $\supp p \cup \supp q = \{k_0, \dots, k_{r-1}\}$. Define by induction a sequence of elements of $S^n$, $\sigma_0, \dots, \sigma_r$ such that $\sigma_i \circ \sigma_{i-1} \circ \cdots \circ \sigma_0(p \restriction k_{i + 1})$ is compatible with $q \restriction k_{i + 1}$. This is done by applying the assumption that $\dot{S}_{k_i}$ is forced to witness the weak homogeneity of $\dot{\Q}_{k_i}$. 

Let $\tilde{p}_i = \sigma_{i-1} \circ \cdots \circ \sigma_0(p)$. Then, $\Vdash \exists \sigma \in \dot{S}_{k_i}, \sigma(\tilde{p}_i(k_i)) \parallel q(k_i)$. By the mixing lemma\footnote{The mixing lemma states that if$\mathbb{R}$ is a forcing notion, $p\in \mathbb{R}$ and $p \Vdash_{\mathbb{R}} \exists x \varphi(x, \dot{y})$ for some $\dot{y}$ and some formula $\varphi$ then there is an $\mathbb{R}$-name $\dot{\tau}$ such that $p \Vdash_{\mathbb{R}} \varphi(\dot\tau, \dot y)$.}, there is an element $\sigma_{i} \in T_{k_i}$ such that $\sigma_{i}(\tilde{p}_i) \restriction k_i = \tilde{p} \restriction k_i$ and $\sigma_{i}(\tilde{p}_i)(k_i)$ is forced to be compatible with $q(k_i)$. Since this automorphism preserves the supports, there are no new coordinates in $\tilde{p}_{i+1}$ below $k_{i+1}$, and the same for $q$ and therefore both restrictions are compatible. 
\end{proof}
Let us remark that the assumption that each automorphism in $S_n$ preserves the names of $S_m$ for all $m \geq n$ entails that $g^{-1} T_m g = T_m$ for all $g \in T_n$, and in particular every element in $S^n$ can be represented as the one which is constructed above. 

Let $T^n$ be the group generated by $\langle T_m \mid m \geq n\rangle$. By the previous claim, $T^0$ witnesses the homogeneity of $\mathbb{P}$. 

Let $\mathcal{F}$ be the filter generated by $\{T^n \mid n <\omega\}$, and note that $T_n \supseteq T_{n+1}$ for all $n$.


By Definition \ref{defn: shuffling poset} and Lemma \ref{lem;the shuffling lemma} (the shuffling lemma), the generic object for $\mathbb{S}(\P, T^0,\mathcal{F})$ will produce the desired generic automorphism. 

This concludes the proof of Claim~\ref{claim:shuffling-omega-iteration}.
\end{proof}

\subsection{Two step iterations}
Here we collect some useful information about the two-step iteration of symmetric extension. See \cite[Section 3.1]{Karagila-iterating-2019} for more details. 
As usual, all our symmetric systems are tenacious, which means that for every condition $p$, the set of automorphisms that fix $p$ is in the filter. 
\begin{defn}\label{def;two-step-symmetric-extension}
Let $\mathcal{S}_0 = \langle \mathbb{P}, \mathcal{F}, \mathcal{G}\rangle$ be a symmetric system and let $\dot{\mathcal{S}}_1$ be a symmetric name for a symmetric system, $\mathcal{S}_1=\langle \mathbb{Q}, \mathcal{K}, \mathcal{H}\rangle$ in the symmetric extension by $\mathcal{S}_0$, denoted by $V^{\mathcal{S}_0}$. 

Let $\mathcal{G} \ast \mathcal{H}$ be the generic semi-direct product of $\mathcal{G} \ast \mathcal{H}$. Namely, the elements of $\mathcal{G}\ast\mathcal{H}$ are names of the form $\langle \pi,\dot{\sigma}\rangle$, where $\pi \in \mathcal{G}$ and $\Vdash \dot\sigma \in \mathcal H$,\footnote{We identify two elements in $\mathcal{H}$ if they are forced by the trivial condition to be the same.
} and the group operation is defined by 
\[\langle \pi_0,\dot{\sigma}_0\rangle * \langle \pi_1,\dot{\sigma}_1\rangle = \langle \pi_0 \pi_1, \dot\sigma_0 \cdot (\pi_0)_*(\dot\sigma_1)\rangle.\] 
Let $\mathcal{F} \ast \mathcal{K}$ be the filter of subgroups generated by the groups of the form 
\[A_{F,\dot{K}}=\{\langle \pi, \dot \sigma\rangle \in \mathcal{G} \ast \mathcal{H} \mid \pi \in F, \Vdash \dot{\sigma} \in \dot{K} \},\] 
For $F \in \mathcal{F}$ and $\Vdash \dot{K} \in \mathcal{K}$ such that $\sym (\dot K) \supseteq F$.

We denote by $\mathcal{S}_0 \ast \mathcal{S}_1$ the symmetric system $\langle \mathbb{P}\ast\mathbb{Q}, \mathcal{G}\ast\mathcal{H}, \mathcal{F}\ast\mathcal{K}\rangle$.
\end{defn}

\begin{lemma}\label{lem;two-step-iteration-of-symmetric-extension}
Let $\mathcal{S}_0, \mathcal{S}_1$ be as in Definition \ref{def;two-step-symmetric-extension}. Let $H$ be a generic filter for $\mathbb{P}$. 
Then $E \in V^{\mathcal{S}_0}$ is a $\mathcal{K}$-symmetric dense open subset of $\mathbb{Q}$ if and only if there is a $\mathcal{F}\ast \mathcal{K}$-symmetric open set $D \subseteq \mathbb{P} \ast \mathbb{Q}$ such that 
\[E = \{\dot{q}^H \mid \langle p, \dot{q}\rangle \in D,\, p \in H\}.\]
\end{lemma}
\begin{proof}
Let $E$ be a dense open set in the symmetric extension. By the definition, it means that there is an $\mathcal{S}_0$-symmetric name $\dot\tau$, $\sym \dot\tau = F \in \mathcal{F}$, such that $\dot{\tau}^H = E$ where $E$ is forced to be a symmetric dense open subset of $\mathbb{Q}$, with respect to the group $\mathcal{H}$. To avoid confusion, when we describe the stabilizer of an element, we will use $\sym_{\mathcal{G}}$ or $\sym_{\mathcal{H}}$ stressing the acting group.

Let us assume also that $p_0 \Vdash_{\P} \sym_{\mathcal H} E = \dot{K}$ for some $\dot{K} \in \dot{\mathcal{K}}$ and let us assume that $\sym_{\mathcal{G}} \dot{K} = F$ as well, by shrinking $F$ and replacing $\dot{K}$ with an equivalent symmetric name, if necessary. By modifying $\tau$ to a name $\tau'$, if necessary, we may assume that every condition $p'$ which is incompatible with $p_0$ forces $E = \mathbb{Q}$ and in particular, $\Vdash_{\P} \sym_{\mathcal{H}} \dot\tau' \supseteq \dot{K}$. We assume also that all members of $F$ fix $p_0$, and thus $\sym_{\mathcal{G}} \dot\tau' \supseteq F$. Without loss of generality, $\dot\tau = \dot\tau'$. 

Let us define: 
\[D = \{\langle p, \dot{q}\rangle \mid p \leq p_0 \vee p \perp p_0,\, p \Vdash_{\P} \dot{q} \in \dot{\tau}\}\]
Following the definition, one can verify that every automorphism of the form $\langle \pi, \dot\sigma\rangle$ such that $\pi \in F$ and $\Vdash_{\P} \dot\sigma \in \dot{K}$ stabilizes $D$ and $D$ is dense open.

The other direction is similar.
\end{proof}
\begin{cor}\label{cor; iterating the shuffling lemma}
Let $\mathcal{S}_0, \mathcal{S}_1$ be as above. Let $G$ be a symmetrically generic filter for $\mathcal{S}_0$, and let $H$ be a symmetrically generic filter for $\mathcal{S}_1$ as in definition \ref{definition:symmetrically-generic-filter}. 
Then, there is an automorphism of $\P\ast\Q$ in a generic extension, that moves $G \ast H$ to a generic filter over $V$ for the iteration $\mathbb{P}\ast\mathbb{Q}$. 

Moreover, for any generic filter $G'\subseteq \P$ which is obtained by an application of a generic automorphism on $G$, there is an automorphism in a further generic extension $\sigma$ on $\Q$ such that for the obtained filter, $H'$, $G' \ast H'$ is generic.  
\end{cor}
\begin{proof}
First, by Lemma \ref{lem;two-step-iteration-of-symmetric-extension}, $G\ast H$ satisfies the requirements of Lemma \ref{lem;the shuffling lemma}, and thus there is a generic automorphism of the two-step iteration sending it to a generic filter for the two-step iteration.

In order to show the moreover part, it suffices to show that there is a projection from $\mathbb{S}(\mathcal{S}_0\ast\mathcal{S}_1)$ to $\mathbb{S}(\mathcal{S}_0)\ast\mathbb{S}(\mathcal{S}^G_1)$ as forcing notions. Note that $\mathbb{S}(\mathcal{S}^G_1)$ is defined in with no reference to the generic of $\mathbb{S}(\mathcal{S}_0)$ and thus
$\mathbb{S}(\mathcal{S}_0)\ast\mathbb{S}(\mathcal{S}^G_1) \cong \mathbb{S}(\mathcal{S}_0)\times\mathbb{S}(\mathcal{S}^G_1)$. \footnote{We will still denote the automorphisms of $\mathbb{Q}$ by $\dot{\sigma}^G$ and similarly for large subgroups of automorphisms, but the dot represents that it is a name with respect to $\mathbb{P}$ and not $\mathbb{S}(\mathcal{S}_0)$.}

Let us define a projection $\rho$ from $\mathbb{S}(\mathcal{S}_0\ast\mathcal{S}_1)$ to $\mathbb{S}(\mathcal{S}_0)\times \mathbb{S}(\mathcal{S}_1^G)$ as follows. The domain of $\rho$ is the dense set of elements of the form $\langle (\pi,\dot\sigma), F\ast \dot{K}\rangle$ such that $\sym \dot\pi_*\sigma, \sym \dot{K} \subseteq F$.

Let \[\rho(\langle (\pi,\dot\sigma), F \ast \dot{K}\rangle) = \langle (\pi, F), ((\pi^{-1}_*\sigma)^G, (\pi_*^{-1}(\dot{K}))^G)\rangle.\] 

Let us show that $\rho$ is order-preserving. Let $(\pi', \sigma')\in F \ast \dot{K}$ and let $F' \ast \dot{K}'\subseteq F \ast \dot{K}$. Let 
\[(\tilde{\pi}, \tilde{\sigma}) = (\pi', \sigma') \cdot (\pi, \sigma) = (\pi'\pi, \sigma' \cdot \pi_*(\sigma))\]
So
\[\rho(\langle (\tilde\pi,\tilde\sigma), (F' \ast \dot{K}')\rangle) = \langle (\pi' \pi, F'), (((\pi' \pi)_*^{-1}(\sigma' \cdot \pi_*(\sigma)))^G, ((\pi' \pi)_*^{-1}(\dot{K}'))^G)\rangle \]

Let us compute
\[\begin{matrix} (((\pi' \pi)_*^{-1}(\dot\sigma' \cdot \pi_*(\dot\sigma))) & = & (\pi' \pi)_*^{-1}(\dot\sigma') \cdot \pi_*^{-1}(\pi')_*^{-1} \pi_*(\dot\sigma) \\ 
& = & (\pi' \pi)_*^{-1}(\dot\sigma') \cdot \dot\sigma
\end{matrix}\]

The last equation is true since $(\pi', \sigma') \in F \ast \dot{K}$, and thus $\pi' \in F \subseteq \sym \dot\pi_*\sigma$. Similarly, $\Vdash \dot\sigma' \in K$ and thus 
\[\Vdash (\pi' \pi)_*^{-1}(\dot\sigma') \in (\pi' \pi)_*^{-1}(\dot{K}) = \pi_*^{-1} (\pi')_*^{-1} (\dot{K}) = \pi_*^{-1} (\dot{K}),\]
so we conclude that $\rho$ is order preserving. Clearly, for every condition $s$, $\{\rho(t) \mid t \leq s\}$ is dense below $\rho(s)$.
\end{proof}

\begin{lemma}
Let $\mathbb{P}$ be a weakly homogeneous forcing notion, as witnessed by $\mathcal G$. Let us assume that for every condition $p\ in \mathbb{P}$, the subgroup $stab(p) = \{\sigma \in \mathcal{G} \mid \sigma(p) = p\}$ witnesses the homogeneity of the cone $\{q \in \mathbb{P} \mid q \leq p\}$. Let $\mathcal{F}$ be a normal filter of subgroups such that $\{stab(p) \mid p \in \mathbb{P}\} \subseteq \mathcal F$ densely. Moreover, let us assume that if $F \in \mathcal F$ and $F \subseteq stab(p)$ then there is $p' \leq p$ such that $stab(p') \subseteq F$.

Let $G \subseteq\mathbb{P}$ be $V$-generic and let $\pi_G$ be a generic autormorphism generated by $\mathbb{S}(\mathbb{P}, \langle \mathcal{G}, \mathcal{F}\rangle)$. Then $\pi^{-1}_G(G)$ and $G$ are mutually generic.
\end{lemma}
\begin{proof}
    Let $D \subseteq \mathbb{P} \times \mathbb{P}$ be dense open, and let $(p, (\pi, F)) \in \mathbb{P} \times \mathbb{S}(\P, \mathcal G, \mathcal F)$. Without loss of generality, $stab(\pi(p)) = F$. Let $(p', q')\leq (p, \pi(p))$ in $D$. So, $\pi(p') \leq \pi(p)$ and thus there is $\sigma \in F$ such that $\sigma(\pi(p'))$ is compatible with $q'$. Let $p'' \leq p'$ be a condition such that $\sigma(\pi(p'')) \leq q'$. 
    
    So $(p'', \sigma\pi(p'')) \in D$. In particular, $(p'', (\sigma\pi, stab(\sigma\pi(p'')) \cap F) )$ forces that $G \times \pi_G(G)$ will meet $D$.
\end{proof}
\begin{cor}
    Let $\mathbb{P}$ be a weakly homogeneous forcing as witnessed by $\mathcal{G}$ and let $\mathcal{F}$ be a normal filter of subgroups of $\mathcal{G}$. Let $G$ be a symmetric filter. Then, for every generic filter $H$ (possibly in an outer model), there is a generic extension in which there is an automorphism $\sigma_H$ such that $\sigma_H(G) = H$.
\end{cor}
\begin{proof}
    Otherwise, there is a condition $h \in H$ forcing that there is no further generic extension introducing an automorphism $\sigma_H$ such that $\sigma_H(G)$. By the weak homogeneity of the forcing $\mathbb{P}$,\footnote{Recall that in our context, a forcing notion $\mathbb{P}$ is weakly homogeneous if for every pair of condition $p,q\in\mathbb{P}$ there are $p' \leq p,\, q' \leq q$ and an automorphism of forcing sending $p'$ to $q'$.} the set $D$ of all conditions $r \in \mathbb{P}$ such that there is $p' \leq h$ and an automorphism $\pi$ in $\mathcal{G}$ such that $\pi(p') \geq r$ is dense open and symmetric (any automorphism that fixes $h$ preserves this set).

    Take $r\in G \cap D$, $h' \leq h$ and $\pi \in \mathcal G$ such that $\pi(h') \geq r$ (and in particular, in $G$). Force an $\mathbb{S}(\mathbb{P},\mathcal{G},\mathcal{F})$-generic filter below the condition $(\pi, F_{h'})$, where $F_{h'}$ is a subgroup in $\mathcal{F}$ of automorphisms that fix $h'$. Let $\pi_G$ be the generic filter. So, $\pi_G^{-1}(G)$ is a $\mathbb{P}$-generic filter containing $h$, contradicting the assumption that $h$ forces that there is no such generic automorphism.
\end{proof}
\section{Construction in $V[c]$ (proof of Theorem~\ref{thm: intermediate extension})}\label{section: construction in Cohen extension}
In \cite{Karagila-Bristol-model-2018}, Karagila constructed an inner model of $L[c]$, where $c$ is a Cohen real, satisfying the failure of the principle $KW_\alpha$ for all $\alpha$. The construction of this so-called \emph{Bristol model} depends on the structure of $L$, in a way that restricts the amount of large cardinals that can exist in the ground model which is suitable for this construction. In this section, we show how to obtain a model with similar properties, as an inner model of $V[c]$, where $V$ is an arbitrary model of $\ZF$ and $c$ is a Cohen-generic real over $V$.

In Section~\ref{subsec: construction in V[c] definitions} we construct a sequence of sets $A_\alpha, T_\alpha$ in $V[c]$ so that the models $M_\alpha=V(A_\alpha, T_\alpha)$ and $M_\infty=V(A, T)$ have the same properties as before. In particular, by corollaries \ref{cor: AC fails in set forcing} and \ref{cor: unbdd KW for V(A, T)}, no set forcing over $M_\infty$ can recover the axiom of choice, and the Kinna-Wagner degree of $M_\infty$ is $\infty$. This will complete the proof of Theorem~\ref{thm: intermediate extension}.

The set $A_1$ will be constructed directly from the Cohen real. For the rest of the construction, we show that for $\alpha\geq 1$, the poset which we used to construct the next stage does not have many dense open subsets inside $M_\alpha$.
(This poset is either $\Q(A_\alpha)\ast\T(\dot{A}_{\alpha+1},A_\alpha)$, or $\B(T_{<\theta})$ for some limit $\theta$.) 

In particular, working in $V[c]$, which is a model of choice extending $M_\alpha$, we will find filters that are sufficiently generic over $M_\alpha$. We then use these generics to continue along the construction.  

We also need to verify that the construction survives through limit stages, as we did in Section~\ref{section:construction}, even though the generics do not come directly from the class iteration $\P$ described before.
Instead, we will show that by ``generically permuting'' the filters, as done in Section \ref{section:shuffling}, we do end up with true generics for the iterations described in Section~\ref{section:construction}, which yields the same models $M_\alpha$, without changing $A_\alpha$ and $T_\alpha$.

The following simple observation will be used repeatedly to confirm that our models satisfy the properties we want, as the models in Section~\ref{section:construction}. Let $N$ be some $\ZF$ extension of $V$, and let $A_\alpha,T_\alpha$ be sets in $N$. We \emph{do not} assume that there is any generic $G$ in $N$ so that $A_\alpha=\dot{A}_0^\alpha[G]$ and $T_\alpha=\dot{T}_0^\alpha[G]$, yet we still want to conclude that the model $V(A_\alpha,T_\alpha)$ satisfies the properties as in Section~\ref{section:construction}.
It suffices to find \emph{some} filter $G$, in \emph{some other} extension of $V$, so that $G$ is generic over $V$ for $\P_0^\alpha$ and the interpretations of $\dot{A}_0^\alpha$ and $\dot{T}_0^\alpha$ according to $G$ are precisely the sets $A_\alpha,T_\alpha$. In this case, the model $V(A_\alpha, T_\alpha)$ as constructed in $N$ is the same as constructed in $V[G]$.
To conclude Theorem~\ref{thm: intermediate extension} we need to prove that such generics $G$ exist, which is done in Section~\ref{subsection: generics in V[c]}.

\subsection{Construction of $A_\alpha, T_\alpha$}\label{subsec: construction in V[c] definitions}
In this subsection, we will describe the construction of the sets $A_\alpha, T_\alpha$ in the model $V[c]$, without proving their properties. Later, in Subsection \ref{subsection: generics in V[c]}, we will verify that the models $V(A_\alpha, T_\alpha)$ satisfy the desired properties.

In order to motivate our indexing scheme, we make the following observation:
\begin{obs}
    Let us assume that $V(A, T)$ is contained in a model of $\AC$, $W$. Then in $W$:
    \begin{enumerate}
        \item If $\delta$ is a limit step, then $|A_\delta| \geq \delta$.
        \item If $\alpha \leq \beta$ then $|A_\alpha|\leq |A_\beta|$.
        \item If $a\in A_\alpha$ and $\langle x_i \mid i \in A_{\alpha + 2}\rangle$ is a sequence of representatives for the equivalence classes of $A_{\alpha+2}$ (namely, $[x_i]=i$), then the sets $\{y\in x_i \mid a \leq y\}$ are all distinct and, in fact, independent.
    \end{enumerate}
\end{obs}
We define an index set $I(\delta)$ for $A_\delta$, from the perspective of $V[c]$. 
\begin{defn}
    We define $I(\delta)$ recursively on the ordinals. 
    \begin{itemize}
        \item $I(0)=\{\langle\rangle\}$. 
        \item $I(\delta + 1) = I(\delta) \times (\omega + \delta)$.
        \item For limit ordinal $\epsilon$, 
        \[I(\epsilon) = \bigcup_{\delta < \epsilon} I(\delta) \times \{\langle 0 \mid i \in \epsilon - \delta\rangle\}\]
    \end{itemize}
\end{defn}
Clearly, $L$ can compute the sequence $I(\delta)$. Note that $|I(\delta)| = |\delta| + \aleph_0$ in $L$ and any extension of it.

For every successor ordinal $\delta$, fix a set $\mathcal{F}_\delta = \langle X_i \mid i \in I(\delta + 1)\rangle\in L$ such that:
\begin{itemize}
    \item For every $\eta \in I(\delta + 1)$, $X_\eta \subseteq \omega + \delta - 1$. 
    \item For every $\eta_0,\dots, \eta_{k-1}, \eta_0',\dots, \eta_{\ell - 1}' \in I(\delta + 1)$ all distinct 
    \[\bigcap_{j < k} X_{\eta_j} \setminus \bigcup_{j < \ell} X_{\eta'_j}\]
    is infinite co-infinite.
\end{itemize}

For example, $\mathcal{F}_1$ is a countable family of subsets of $\omega$ so that any finite boolean combination of sets in $\mathcal{F}_1$ is infinite. At stage $\delta$, we want such a family of size $|I(\delta+1)|=|\delta|$, so the members of $\mathcal{F}_\delta$ need to be subsets of a set of size at least $|\delta|$. For this reason, we take the members of $\mathcal{F}_\delta$ to be subsets of $\omega+\delta$. The $\omega$ is only there to take care of the cases where $\delta$ is finite. 

We would like the construction of $A_\alpha, T_\alpha$ from $c$ to be as absolute as possible. That is, it will depend only on $c$ and the ordinals and not on the particular universe in which the construction is carried (for example, $L[c]$, $V[c]$, or a generic extension of $V[c]$). For this reason, we picked all the above elements in $L$.

We will construct recuresivelt sets $x^\alpha_i$ for $i\in I(\alpha)$ such that for $\alpha$ successor $A_\alpha = \{[x^\alpha_i] \mid i \in I(\alpha)\}$ and for $\alpha$ limit $A_\alpha = \{x^\alpha_i \mid i \in I(\alpha)\}$.

First, we define $A_1$. We may write $c$ as $\seqq{c_n}{n<\omega}\in (2^\omega)^\omega$, a $\Q(\omega)$-generic over $M_0=V$.
Define $x^1_{\langle n\rangle}=c_n$, $A_1(n)=[c_n]=\set{y\subset \omega}{c_n\Delta y\textrm{ is finite}}$, and $A_1=\set{A_1(n)}{n\in\omega}$. 

Next, assume the inductive hypothesis at $\beta$ and prove it for $\alpha=\beta + 1$. 

If $\beta$ is a successor ordinal, for $\eta\in I(\beta+1)$ define
\[x^\alpha_{\eta} = \{[x^\beta_{{\eta'} ^\smallfrown \langle \gamma\rangle}] \mid \gamma \in X_{\eta},\, \eta'\in I(\beta - 1)\}\]

If $\beta$ is a limit ordinal, for $\eta \in I(\beta+1)$, write $\eta = \bar{\eta}\smallfrown\langle\rho\rangle$ where $\bar{\eta}\in I(\beta)$ and $\rho < \omega+\beta$. Define
\[x^\alpha_{\eta} = \{x^\beta_{{\eta'} ^\smallfrown \langle \gamma\rangle ^\smallfrown \langle 0 \rangle^{\beta - \delta}} \mid \delta < \beta \text{ successor},\,\gamma \in X_{\bar{\eta} \restriction \delta ^\smallfrown \langle \rho\rangle} \setminus \{0\},\, \rho < \omega + \delta,\, \eta' \in I(\delta)\}\]

Finally, for $\eta\in I(\beta+1)$, $\eta = \bar{\eta}\smallfrown \langle \rho \rangle$ where $\bar{\eta}\in I(\beta)$, define
\begin{equation*}
    \pi_\beta([x_\eta^\alpha])=[x_{\bar{\eta}}^\beta].
\end{equation*}
That is, we put each $[x_{\bar{\eta}\smallfrown\langle\rho\rangle}^\beta]$ above $[x_{\bar{\eta}}^\beta]$ in the tree.

Next consider a limit stage $\alpha$, assuming that we have constructed all the above objects at all levels below $\alpha$. For each $\eta \in I(\alpha)$, define
\begin{equation*}
x^\alpha_\eta = \{[x_{\eta \restriction \gamma}] \mid \gamma < \alpha \text{ successor}\} \cup \{x_{\eta \restriction \gamma} \mid \gamma < \alpha\text{ limit}\},
\end{equation*}
a branch in the tree $T_{<\alpha}$. The tree structure is determined: each branch extends all of its members.

\begin{lemma}\label{lem; x n alpha are independent}
For every $\beta$, $\{x_{\eta}^{\beta+1} \mid \eta \in I(\beta+1)\}$ is an independent family. Moreover, for every $b \in A_{<\beta}$, $\{x_\eta^{\beta+1} \cap \{y \in A_{\beta} \mid b \leq y\} \mid \eta \in I(\beta+1)\}$ is an independent family.  \end{lemma}
\begin{proof}
For $\beta$ successor, it is clear from the construction.

For limit ordinal $\beta$, let $\eta_0,\dots, \eta_{\ell - 1}\in I(\beta+1)$. Each $\eta_i$ is of the form $\bar{\eta}_i^\smallfrown \langle \rho_i\rangle$. Taking a sufficiently large $\delta < \beta$ such that $\rho_i < \omega + \delta$ for all $i$ and if $\eta_i \restriction \beta \neq \eta_j \restriction \beta$, then this inequality already occurs below $\delta$. So, $x^{\beta+1}_{\eta_i}$ contains the set 

\[E_i := E_i(\eta',\delta) = \{x^\beta_{{\eta'}^\smallfrown \langle \gamma\rangle ^\smallfrown \langle 0 \rangle^{\beta - \delta}} \mid \gamma \in X_{\eta_i \restriction \delta ^\smallfrown \langle \rho_i\rangle} \setminus \{0\}\}\]
for any fixed $\eta' \in I(\delta)$. The collection $\{E_i \mid i < \ell\}$ is an independent family. Moreover, $x_{\eta_i}^{\beta+1} \cap (\bigcup_{j < \ell} E_j) = E_i$. One can verify that by noticing that an element of the form $x:=x^\beta_\zeta \in x^{\beta+1}_{\eta_i}$ must satisfy that the last ordinal $\delta'$ for which $\zeta(\delta')\neq 0$ is the unique ordinal such that $x \in E_i(\zeta\restriction \delta', \delta')$. \end{proof}
\begin{lemma}\label{lemma: const in V[c] branches are dense}
    For every $t\in A_{<\alpha}$ there is a branch in $A_\alpha$ containing it.
\end{lemma}
\begin{proof}
    Let $t = [x_{\eta}]$ for $\eta \in I(\gamma)$, $\gamma < \alpha$. By the definition of $I(\alpha)$, $\eta = {\eta'} ^\smallfrown \langle 0 \rangle^{\alpha - \gamma} \in I(\alpha)$, and thus $x_{\eta} \in A_\alpha$ is a branch containing $t$.
\end{proof}

\subsection{Genericity}\label{subsection: generics in V[c]}
In this subsection we will prove that for every $\alpha$ there is a generic filter $G\subseteq \mathbb{P}_\alpha$, external to $V[c]$, such that the sets $A_\alpha, T_\alpha$ which we defined in the previous subsection are $\dot{A}_\alpha[G], \dot{T}_\alpha[G]$. 

During our induction, we will always assume that $\alpha$ is  \emph{countable}. This does not restrict the construction, as there is a forcing extension of $V[c]$ in which $\alpha$ is countable and thus we may, without loss of generality, work in this extension. We will implicitly re-enumerate the members of $A_\alpha$ in a sequence of order-type $\omega$ using the countablility hypothesis, and write $\{x^\alpha_m \mid m <\omega\}$ instead of $\{x^\alpha_\eta \mid \eta\in I(\alpha)\}$ in order to define a filter for $\mathbb{P}_\alpha$. 

We prove by induction on $\theta$ the following slightly stronger statement.  

There is a collection $\mathcal{T}$ of $V$-generic filters such that
\begin{enumerate}
    \item Every $G \in \mathcal{T}$ is a generic filter for $\mathbb{P}_\alpha$ for some $\alpha < \theta$.
    Moreover, $A_\alpha = \dot{A}_\alpha[G]$ and $\dot{T}_\alpha = \dot{T}_\alpha[G]$.
    \item $\{1_{\mathbb{P}_0}\} \in \mathcal{T}$, i.e.\ the root of the tree is the generic for the trivial forcing $\mathbb{P}_0$.
    \item If $G \in \mathcal{G}$ is generic for $\mathbb{P}_\alpha$ and $\beta < \alpha$, then $G \restriction \beta \in \mathcal{T}$.
    \item For every $G \in \mathcal{T}$ which is generic for $\mathbb{P}_\alpha$ and for every $\alpha < \beta < \theta$ there is a generic filter for $\mathbb{P}_\beta$, $H \in \mathcal{T}$ such that $H \restriction \alpha = G$.
    \item $\mathcal{T}$ is closed under automorphisms of $\mathbb{P}_\alpha$ from $V$ that fix $\dot{A}_\alpha, \dot{T}_\alpha$ for all $\alpha$. 
\end{enumerate}

Equivalently, we are going to construct a normal tree of generic filters for the iterations. Using the normality of the tree, we will be able to find at limit steps of countable cofinality always a filter which is suitable for Claim \ref{claim:shuffling-omega-iteration}. 

At successor steps of the construction, we will only add elements to the top level of the tree, without modifying the previous level. This is done by analyzing the filter $G$ which we use in order to construct $A_{\alpha+1}$ and $T_{\alpha+1}$. $G$ can be viewed as a filter for $\mathbb{P}_{\alpha + 1}$, which we decompose as a two-step iteration $\mathbb{P}_\alpha \ast \mathbb{W}_\alpha$ (where $\mathbb{W}_\alpha$ is either $\mathbb{Q}_\alpha \ast \mathbb{R}_\alpha$, if $\alpha$ is a successor ordinal or $\mathbb{B}_\alpha$ if $\alpha$ is limit). 

We look at $M_\alpha$: this model is obtained using $G \restriction \alpha$. By the inductive hypothesis, there is a generic automorphism that does not move the interpretation of $A_\alpha$ and $T_\alpha$ but moves $G$ to a generic filter over $V$. So, $M_\alpha$ satisfies indiscernibility, as proved in Section \ref{section:construction}. We will analyze the dense open sets from $M_\alpha$ and show that the conditions for Lemma \ref{lem;two-step-iteration-of-symmetric-extension} hold and thus we may find another generic automorphism $\pi$ such that $\pi^{-1}(G)$ is a $\mathbb{P}_{\alpha+1}$-generic for $V$. By the "moreover" part of Lemma \ref{lem;two-step-iteration-of-symmetric-extension}, we may assume that the restriction $\pi^{-1}(G) \restriction \alpha$ is the original generic that we started from, and that it is already in the tree $\mathcal{T}$.

At successor steps, we are going to find generic filters for $M_\alpha$ and then apply Corollary \ref{cor; iterating the shuffling lemma} to get to the next step in the tree.

In order to apply Lemma \ref{lem;the shuffling lemma}, we must specify a filter of subgroups of $\mathcal{G}_\alpha$. Let $\mathcal{F}_\alpha$ be the filter generated by \[\{\sym \dot x \mid \dot x \text{ is the canonical name for an element in} \trcl A_\alpha\}\]

\subsubsection{Adding $A_{\alpha+1}$}
Let $\mathbb{Q}_n(W)$ be the restriction of $\mathbb{Q}(W)$ to conditions whose domain is a subset of $n\times W$. 
Recall that $\dot{x}_n$ is the $\mathbb{Q}(W)$-name for the set $\set{w\in W}{\dot{G}(n,w)=1}$. Below we identify $\dot{x}_n$ with a $\mathbb{Q}_{k}(W)$-name as well, for $k>n$, in the natural way. 

Suppose we proved the inductive hypothesis for all ordinals below $\alpha$, and let us assume that $\alpha$ is a successor ordinal. In order to prove the claim for $\alpha$ we show first that the filter defined by $\langle x_n^\alpha \mid n < \omega\rangle$ meets every dense open set in $M_\alpha$. As any such set is definable using finitely many elements in $\trcl(A_\alpha)$, any automorphism of $\mathbb{Q}_\alpha \ast \mathbb{R}_\alpha$ from $\mathcal{G}_\alpha$ that fixes those elements will fix $D$. We conclude that $D = X \cdot D$ for some $X \in \mathcal{F}_{\alpha+1}$.  

Then, we will use Lemma \ref{lem;the shuffling lemma} in order to get for every $G_\alpha \subseteq \mathbb{P}_\alpha$ generic that induces $A_{<\alpha}$, $T_{<\alpha}$ a $V[G_\alpha]$-generic $G = G^\alpha$ inducing $A_\alpha$, $T_\alpha$.  

\begin{lemma}
Suppose $G\subset\mathbb{Q}(A_\alpha)$ satisfies that for each $n<\omega$, $G\cap \mathbb{Q}_n(A_\alpha)$ is generic over $V(A_\alpha,T_\alpha)$.
Then there is a filter $G'\subset\mathbb{Q}(A_\alpha)$, in some generic extension, which is generic over $M_\alpha$ and such that the symmetric difference of $\dot{x}_n[G']$ and $\dot{x}_n[G\cap\mathbb{Q}_n(A_\alpha)]$ is finite.
\end{lemma}
\begin{proof}
Let us view $\Q(A_\alpha)$ as the finite support product of $\omega$ many copies of $\Q_1(A_\alpha)$.
Force, by finite conditions, a generic sequence of finite support permutations of $A_\alpha$, $\vec{s}=\seqq{s_n}{n<\omega}$, and define $G'(n, a)=1\iff G(n,s_n(a))=1$.
By Lemma~\ref{lem;the shuffling lemma}, $G'$ is generic over $V(A_\alpha,T_\alpha)$.
\end{proof}
\begin{remark}\label{remark: countable successor stages shuffling is Cohen}
If $N=V[c]$ and $\alpha$ is a countable ordinal, then the forcing to add $\vec{s}$ is isomorphic to adding a single Cohen real, and we can find such generic Cohen reals over $V(A_\alpha, T_\alpha)$ in $V[c]$. It follows that $G'$ can be found in $V[c]$, since in our definition of the derivation of $A_\beta, T_\beta$ there is always some Cohen real in $V[c]$ which is generic with respect to the whole construction. 
\end{remark}

\begin{lemma}
Let $\alpha > 1$ and let $n < \omega$. Let $G_n$ be the filter generated by $\{x_m^\alpha \mid m < n\}$. Then $G$ is $M_\alpha$-generic for $\mathbb{Q}_n(A_\alpha)$.
\end{lemma}
\begin{proof}
In order to prove this statement, let us analyze the dense open subsets of the forcing $\mathbb{Q}_n(A_\alpha)$ in $V(A_\alpha, T_\alpha)$.
\begin{claim}\label{claim: Q(A) genericity}
The filter of dense open subsets of $\mathbb{Q}_n(A_\alpha)$ is generated in $V(A_\alpha,T_\alpha)$ by:
\begin{itemize}
    \item For $m < n, a \in A_\alpha$, $D^0_{m,a} = \{p \mid (m, a) \in \dom p\}$.
    \item For $t \in A_{<\alpha}$, $b \colon n \to 2$, $k < \omega$, let 
    $D^1_{t,b,k}$ be the set of all $p$ such that there are $a_0, \dots, a_{k-1} \in A_\alpha$, distinct and above $t$, and $\forall m < n, p(m, a_i) =b(m)$.
    \end{itemize}
\end{claim}
\begin{proof}
Let $D \in V(A_\alpha, T_\alpha)$ be dense open in $\mathbb{Q}_n(A_\alpha)$. By Lemma \ref{thm: full indiscernibility}, there is a finite set of parameters that define $D$, $\bar{p} \in \trcl(A_\alpha)$. Let $\varphi$ be a formula such that:
\[V(A,T) \models D = \{x \mid \varphi(x, \bar{p}, A_\alpha, T_\alpha, v)\},\]
where $v\in V$. Without loss of generality, $\varphi$ determines the full type of $\bar{p}$. Let us show that we can assume that $\bar{p}$ is of the form $\bar{a} ^\smallfrown \bar{t}$ where $\bar{a} \in A_\alpha^{<\omega}$ and $\bar{t} \in A_{<\alpha}^{<\omega}$. Indeed, if $\bar{p} = \bar{p}^{\prime\smallfrown} u$ where $u \in d \in A_{\beta}$, $\beta \leq \alpha$, then let us define
\[D' = \{q \in \mathbb{Q}_n(A_\alpha) \mid \exists u' \in d, \varphi(q, \bar{p}^{\prime\smallfrown} u', A_\alpha, T_\alpha, v)\}.\]
Clearly, $D'$ is definable using $\bar{p} ^{\prime\smallfrown} d$ so it is enough to show that $D' = D$. Once we show that, then by repeating this process finitely many times we can replace all such parameters from $\bar{p}$ with their equivalence class (which is in $\bigcup_{\beta < \alpha} A_\beta$). 

Note that $D \subseteq D'$. Let $q \in D'$ and let $u'$ witness that. The condition $q$ is definable (explicitly) using finitely many parameters from $A_\alpha$, $\bar{a}_q$. By indiscernibility, the membership of $q$ to $D$ depends only on the type of $\bar{a}_q \cup \bar{p}$. Since changing $u$ to $u'$ does not change this type, using our assumption on $\varphi$, and since $\bar{a}_q \subseteq A_\alpha$ this truth value remains the same.  

As before, we can modify $\bar{t}$ so it is nonempty and all its elements have the same height. 

Our strategy to prove the claim is to fix a set of parameters $\bar{a}, \bar{t}$ defining $D$ as above, and for each possible $p \colon n \times \bar{a} \to 2$, we find a dense open set which is an intersection of length $\ell^p$ of sets of the form $D^1_{t^p_i, b^p_i, k^p_i}$ such that if $q \in \bigcap_{i<\ell^p} D^1_{t^p_i,b^p_i,k^p_i}$ and $q \leq p$ then $q \in D$. Then,
\[D\supseteq \bigcap_{k < n, a \in \bar{a}}  D^0_{k, a} \cap \bigcap_{p \colon n \times \bar{a} \to 2}\bigcap_{i<\ell^p} D^1_{t^p_i,b^p_i,k^p_i}\]

Fix $p$ to be a condition with domain $n \times \bar{a}$ and let $q \in D$, $q \leq p$. We may assume that $\dom q = n \times \bar{a}_q$. Let $\bar{t}'$ be an extension of $\bar{t}$ such that for every $a$ such that $(m, a)\in \dom q$ for some $m < n$, there is some $r \in \bar{t}'$, $r\leq_T a$, and the level of $r$ is the same as the level of the elements of $\bar{t}$. Let $\bar{a}_q \in A^{<\omega}$ be an enumeration of $\{b \mid \exists m < n (m,b) \in \dom q\}$ and let $\psi$ be a formula that defines $q$ from $\bar{a}_q$ in a natural way and also
describes the full type of $\bar{a}_q \cup \bar{t}' \cup \bar{a}$. One can verify that the type of $\bar{a}_q$ over $\bar{a} ^\smallfrown \bar{t}$ can be computed from the number of elements in $\bar{a}_q \setminus \bar{a}$ in each cone. Thus, there is a sequence $\{\langle t_i, b_i, k_i\rangle \mid i < \ell^p\}$ such that $q \in \bigcap D^1_{t_i, b_i, k_i}$ and for every other condition $q' \in \bigcap D^1_{t_i, b_i, k_i}$ such that $q' \leq p$, and $q'$ is stronger than a condition $q''$ such that $\psi(q'', \bar{a}_{q''})$ holds and the type of $\bar{a}_{q''}$ over $\bar{a} \cup \bar{t}'$ is the same as the type of $\bar{a}_q$.
By indiscernibility, $q'' \in D$, and so $q'\in D$, as wanted.


\end{proof}
Thus, we conclude that a collection of $n$ subsets of $A_\alpha$ is $V(A_\alpha, T_\alpha)$-generic for $\mathbb{Q}_n(A_\alpha)$ if and only if their intersection with each cone above $t \in A_{<\alpha}$ is $n$ independent subsets.

Since this happens at $\langle x_m^\alpha \mid m < n\rangle$ by construction, we conclude that $G_n$ is $M_\alpha$-generic.


\end{proof}

\subsubsection{Adding $T_{\alpha+1}$}
As a corollary of the indiscernibility result, it is quite easy to find $\T(A_{\alpha+1}, A_\alpha)$-generics over $V(A_{\alpha+1}, T_\alpha)$:
\begin{lemma}
Let $f\colon A_{\alpha+1}\to A_\alpha$ be onto and with every fiber infinite. Define $F = \{p \in \T(A_{\alpha+1},A_\alpha) \mid p \subseteq f\}$. Then $F$ is $\T(A_{\alpha+1},A_\alpha)$-generic over $V(A_{\alpha+1},T_\alpha)$.
\end{lemma}
\begin{proof}
Let $D$ be a dense open set in $V(A_{\alpha+1}, T_\alpha)$. As in the proof of Claim~\ref{claim: Q(A) genericity}, using the indiscernibility we may assume that $D$ is definable using a formula $\varphi$ and only parameters $\bar{a}, \bar{b}, \bar{x}$, where $\bar{a} \in A_{\leq\alpha}^{<\omega}$, $\bar b \in A_{\alpha+1}^{<\omega}$, and $\bar x \subseteq \trcl A_{\alpha+1}$, and a parameter from $V$. We may also assume that $[x_i]$ appears in $\bar{b}$ for every $x_i \in \bar{x}$.

Pick $p \in F$ such that $\dom p \supseteq \bar{b}$ and $\range p \supseteq \bar{a}$. Let $q \leq p$ be a condition in $D$ and let $\bar{b}' \supseteq \bar{b}$ be the domain of $q$ and $\bar{a}' \supseteq \bar{a}$ be the range of $q$. Let us look at the type of $\bar{b}' \cup \bar{a}'$ over $\bar{a},\bar{b},\bar{x}$.

For each $a \in \bar{a}'$, let $n_a$ be the number of elements $b\in \bar{b}' \setminus \bar{b}$ such that $q(b) = a$. 
\begin{claim}
Let $q' \leq p$ be a condition such that $\range q' = \bar{a}'$ and for every $a\in \bar{a}'$, $|\{b \in \dom q' \setminus \bar{b} \mid q'(b) = a\}| = n_a$. Then $q' \in D$. 
\end{claim}
\begin{proof}
The condition $q$ is definable using an explicit formula $\psi$ and the parameters $\bar{b}', \bar{a}'$, where the formula $\psi$ simply states which element is being sent to which element. By our assumption on $q'$, there is an enumeration of $\bar{b}'' = \dom q'$ such that $q'$ is definable using the same $\psi$ and parameters $\bar{b}''$ and $\bar{a}'$. Since the type of $\bar{b}' \cup \bar{a}'$ over $\bar{a},\bar{b},\bar{x}$. is the same as the type of $\bar{b}'' \cup \bar{a}'$ over $\bar{a},\bar{b},\bar{x}$., we apply indiscernibility and conclude that $q' \in D$.  
\end{proof}
Finally, since $f$ is surjective with infinite fibers, there is $q' \leq p$ as in the claim in $F$.
\end{proof}
By construction, $\pi_{\alpha}$ is a surjective function from $A_{\alpha+1}$ to $A_\alpha$, with the property that every fiber is infinite, so the filter $F = \{p \in \T(A_{\alpha+1}, A_\alpha) \mid p \subseteq \pi_\alpha\}$ is generic over $V(A_{\alpha+1}, T_{\alpha})$.

Assume now that $G\subset \mathbb{P}_\alpha\ast\Q(\dot{A}_\alpha)$ is generic over $V$ which computes the correct $A_\alpha,T_\alpha$ and $A_{\alpha+1}$. We want to find a filter $G'\ast F'$ which is generic over $V$ for $\mathbb{P}_\alpha\ast\Q(\dot{A}_\alpha)\ast\mathbb{T}(\dot{A}_{\alpha+1},\dot{A_\alpha})$.

This is obtained by Corollary \ref{cor; iterating the shuffling lemma}. Indeed, we may assume (by genericity) that $G$ was obtained by applying a generic permutation on the symmetrically generic filter for $\mathbb{P}_\alpha \ast \Q(\dot{A}_\alpha)$ and thus we can force with the quotient forcing for adding the generic permutation for $\mathbb{T}(\dot{A_{\alpha+1}}, \dot{A_\alpha})$ and obtain a generic iteration of the two-step iteration that moves the symmetrically generic filter to an actual $V$-generic filter.  

\subsection{Limit step}\label{subsection: V[c] limit stage}

\begin{lemma}
Let $\theta$ be a countable ordinal and let us assume that the inductive hypothesis holds below $\theta$. Assume further that in $V[c]$ there is a well-ordering of a large enough initial segment, to apply Claim~\ref{claim:shuffling-omega-iteration}. Then the inductive hypothesis holds for $\theta$.
\end{lemma}
\begin{proof}
By assumption $\theta$ is countable. Let $\langle \alpha_n \mid n < \omega\rangle$ be an increasing cofinal sequence at $\theta$. We may think of the poset $\P_0^\alpha$ as an $\omega$-length finite support iteration, with the $n$-th step being $\mathbb{P}_{\alpha_n}^{\alpha_{n+1}}$. 
Recall from Section~\ref{subsubsection: permutations} that for each $\mathbb{P}_{\alpha_n}^{\alpha_{n+1}}$ there is a subgroup $S_n$ of automorphisms of the poset such that
\begin{itemize}
    \item for any $p,q$ there is $s\in S_n$ sending $p$ to be compatible with $q$;
    \item any $s\in S_n$ fixes $\dot{A}_\alpha,\dot{T}_\alpha$ for all $\alpha$.
\end{itemize}
These properties verify the hypothesis of Claim~\ref{claim:shuffling-omega-iteration}. We conclude from the claim that there is a ``shuffled'' filter $\tilde{G}$ (in some further generic extension) so $A_\alpha=\dot{A}_0^\alpha[\tilde{G}]$ and $T_\alpha=\dot{T}_0^\alpha[\tilde{G}]$ for any $\alpha<\theta$.
\end{proof}
\begin{lemma}
Let $\theta$ be arbitrary. If the inductive hypothesis holds below $\theta$ then it holds at $\theta$. 
\end{lemma}
\begin{proof}
Forcing with a large enough collapse $G$ over $V[c]$ so that in $V[c][G]$, $\theta$ is countable and we have enough choice to apply Claim~\ref{claim:shuffling-omega-iteration}. Note that the definition of $A_\alpha, T_\alpha$ is the same in $V[G]$ and $V[G][c]$, since the recipe for constructing them is absolute. Now we can apply the previous lemma.
\end{proof}


\subsection{Tree step}
Let $\theta$ be a limit ordinal. Suppose $M_{<\theta}=V(A_{<\theta},T_{<\theta})$ has been constructed.
Recall the definitions of $\B$ from Section~\ref{subsection: tree step}, adding by finite approximations $\omega$ many branches through $T_{<\theta}$.

To find a generic for $\mathbb{B}$ over $M_{<\theta}$, we will instead use the following slightly different $\tilde{\B}$, which produces the same symmetric extensions as $\mathbb{B}$.
\begin{defn}
    Let $\tilde{\B}$ be the poset whose conditions are partial finite functions from $\theta$ to $A_{<\theta}$, ordered by the reverse ordering of the tree $T_{<\theta}$ on each coordinate. Given a generic $\tilde{G}\subseteq \tilde{\mathbb{B}}$, which we identify with a $\theta$-sequence of generic branches, let $A_\theta^{\tilde{G}}$ be the unordered set of these branches.
\end{defn}
\begin{lemma}\label{lemma:generic-for-tree-steps}
    Let $\tilde{G} \subseteq \tilde{\B}$ be $V(A_{<\theta}, T_{<\theta})$-generic. There is, in a further generic extension, a $V(A_{<\theta}, T_{<\theta})$-generic filter $G\subseteq \B$ so that $A_\theta^{G} = A_{\theta}^{\tilde{G}}$.
\end{lemma}
\begin{proof}
Simply note that after collapsing $\theta$ to be countable $\mathbb{B}$ and $\tilde{\mathbb{B}}$ are forcing isomorphic.
\end{proof}

First we note that in $M_{<\theta}$ there are very few symmetric dense open subsets of $\tilde{\mathbb{B}}$




\begin{lemma}\label{lemma: pseudo-generic limits}
Let $\mathcal{S} = \langle \tilde{\B}, \mathcal{G}, \mathcal{F}\rangle$, where $\mathcal{G}$ is the group generated by $b_\sigma$ and $\mathcal{F}$ is the filter of subgroups generated by $\langle \{b_\sigma \mid \sigma \restriction a = id_a\} \mid a \subseteq \omega,\text{ finite}\rangle$. 

A filter $\tilde{G}$ is symmetrically generic for $\mathcal{S}$ in $M_{<\theta}$ if and only if 
\begin{enumerate}
    \item Every branch defined by $\tilde{G}$ is cofinal, namely for every $\zeta, \zeta' < \theta$, there is an element in the $\zeta$-th coordinate of $\tilde{G}$ of level $\zeta'$ in the tree.
    \item Every pair of branches is different.
    \item For every $t \in A_{<\theta}$ there is a branch in $\tilde{G}$ that contains $t$.
\end{enumerate}
\end{lemma}
\begin{proof}
Let $D\subseteq \B$ be a symmetrically dense open set in $V(A_{<\theta}, T_{<\theta})$ (with respect to the group of finite permutations of coordinates). Let $\tilde{G}$ be a filter that meets the dense open sets as in the statement of the lemma. 
Fix $a_0, \dots, a_{m-1} \in A_{<\alpha}$ be such that $D$ is definable from $a_0, \dots, a_{m-1}, A_{<\alpha}, T_{<\alpha}$ and a parameter in $V$. 

Let $p \in \tilde{G}$ be a condition such that $\dom p$ is the range of $q$ and that all its elements are in a level higher than all the $a_i$'s. 
Let $p' \leq p$ be in $D$. 
Let $t_0, \dots, t_{n-1}$ be the elements of $p'$. Assume without loss of generality that the level of each $t_k$ is greater than the level of each $a_i$.

By the tree indescirnibility, the truth value of $p' \in D$ is a function of the type of $t_0,\dots, t_{n-1}$ over $a_0,\dots, a_{m-1}$. Since the levels of $t_i$ are (without loss of generality) much higher than the parameters in $\vec{a}$ and all are different and in the same level, the diagram is determined by the level and the values of $a_j \leq_T t_i$ for all $i, j$. In particular, $p' \in D$ if and only if $p \cup r \in D$ where $r = p' \restriction (\dom p' \setminus \dom p)$. By the assumptions of the lemma, there are branches in $\tilde{G}$ which contain the elements of $r$. Let $p'' \leq p'$ be a condition in $\tilde{G}$ that contains those elements (note that we may move the domain, as $D$ is invariant under a large group of permutations of the domain), then $p''$ meets $D$. 
\end{proof}
Let $\tilde{G}\subset \tilde{\mathbb{B}}$ be the filter generated by the set of branches $A_\theta$, enumerated by $\theta$, constructed in Section~\ref{subsec: construction in V[c] definitions}. Note that $\tilde{G}\subset\tilde{\B}$ satisfies conditions (1),(2),(3) above. 

Finally, given a generic $G_{<\theta}$ for the iteration $\mathbb{P}_{<\theta}$, we conclude from Corollary \ref{cor; iterating the shuffling lemma} that after applying a generic automorphism of $\mathbb{P_\theta} = \mathbb{P}_{<\theta}\ast \tilde{\mathbb{B}}$ to $G_{<\theta}\ast \tilde{G}$, we get a generic filter over $V$. This concludes the proof that the inductive hypothesis can be extended.


\end{document}